\documentclass{article}
\usepackage{amssymb,amsfonts,amsmath,amsthm,amsopn,amstext,amscd,latexsym,xy,color,stmaryrd}
\usepackage{hyperref}
\usepackage{tikz,multicol,graphicx,longtable}
\usetikzlibrary{arrows,positioning,decorations.pathmorphing,
  decorations.markings
} 
\tikzset{inner sep=0pt, 
  root/.style={circle,draw,minimum size=7pt,thick}, 
  fatroot/.style={circle,draw,minimum size=10pt,thick}, 
  short root/.style={circle,fill,minimum size=7pt}, 
  doublearrow/.style={postaction={decorate}, 
  decoration={markings,mark=at position .7
  with {\arrow{angle 60}}},double distance=3pt,thick}
} 
\theoremstyle{plain}
\input xy
\xyoption{all}

\newtheorem{theorem}{Theorem}[section]
\newtheorem{lemma}[theorem]{Lemma}
\newtheorem{proposition}[theorem]{Proposition}
\newtheorem{corollary}[theorem]{Corollary}
\newtheorem{definition}[theorem]{Definition}

\theoremstyle{remark}
\newtheorem{remark}[theorem]{Remark}
 \numberwithin{equation}{section}

\def\<{\langle}
\def\>{\rangle}

\DeclareMathOperator{\cha}{char}
\DeclareMathOperator{\Frac}{Frac}
\DeclareMathOperator{\Hom}{Hom}

\DeclareMathOperator{\Ad}{Ad}
\DeclareMathOperator{\rank}{rank}

\DeclareMathOperator{\Gal}{Gal}
\DeclareMathOperator{\gr}{gr}
\DeclareMathOperator{\Aut}{Aut}

\DeclareMathOperator{\Sel}{Sel}
\DeclareMathOperator{\ad}{ad}

\DeclareMathOperator{\ord}{ord}
\DeclareMathOperator{\disc}{disc}

\DeclareMathOperator{\Lie}{Lie}

\DeclareMathOperator{\Pic}{Pic}

\DeclareMathOperator{\inv}{inv}

\DeclareMathOperator{\diag}{diag}

\DeclareMathOperator{\im}{im}
\DeclareMathOperator{\Spec}{Spec}

\newcommand{\cC}{{\mathcal C}}

\newcommand{\cE}{{\mathcal E}}
\newcommand{\cF}{{\mathcal F}}

\newcommand{\cJ}{{\mathcal J}}

\newcommand{\cL}{{\mathcal L}}
\newcommand{\cM}{{\mathcal M}}

\newcommand{\cO}{{\mathcal O}}
\newcommand{\cP}{{\mathcal P}}
\newcommand{\cQ}{{\mathcal Q}}

\newcommand{\cV}{{\mathcal V}}

\newcommand{\cX}{{\mathcal X}}
\newcommand{\cY}{{\mathcal Y}}

\newcommand{\frc}{{\mathfrak c}}

\newcommand{\frg}{{\mathfrak g}}
\newcommand{\frh}{{\mathfrak h}}

\newcommand{\frl}{{\mathfrak l}}
\newcommand{\ffrm}{{\mathfrak m}}

\newcommand{\frp}{{\mathfrak p}}

\newcommand{\frs}{{\mathfrak s}}
\newcommand{\frt}{{\mathfrak t}}

\newcommand{\frz}{{\mathfrak z}}
\newcommand{\bbA}{{\mathbb A}}

\newcommand{\bbF}{{\mathbb F}}
\newcommand{\bbG}{{\mathbb G}}

\newcommand{\bbP}{{\mathbb P}}
\newcommand{\bbQ}{{\mathbb Q}}
\newcommand{\bbR}{{\mathbb R}}

\newcommand{\bbZ}{{\mathbb Z}}

\newcommand{\GL}{{\mathrm{GL}}}
\newcommand{\SL}{{\mathrm{SL}}}

\newcommand{\SO}{{\mathrm{SO}}}

\newcommand{\PGL}{{\mathrm{PGL}}}

\title{On the average number of 2-Selmer elements of elliptic curves over $\bbF_q(X)$ with two marked points}
\author{Jack A. Thorne\footnote{This research was partially conducted during the period the author served as a Clay Research Fellow.}}

\begin{document}
\maketitle
\begin{abstract}
We consider elliptic curves over global fields of positive characteristic with two distinct marked non-trivial rational points. Restricting to a certain subfamily of the universal one, we show that the average size of the 2-Selmer groups of these curves exists, in a natural sense, and equals 12. Along the way, we consider a map from these 2-Selmer groups to the moduli space of $G$-torsors over an algebraic curve, where $G$ is isogenous to $\SL_2^4$, and show that the images of 2-Selmer elements under this map become equidistributed in the limit.
\end{abstract}
\tableofcontents
\section{Introduction}\label{sec_introduction}

Let $K$ be a global field. To any elliptic curve $E / K$ and integer $n \geq 1$ not dividing the characteristic of $K$, one can attach the $n$-Selmer group
\[ \Sel_n(E) = \ker(H^1(K, E[n]) \to \prod_v H^1( K_v, E)). \]
The cohomology groups here are Galois cohomology, and the product is over the set of all places $v$ of the global field $K$. The $n$-Selmer group then fits into a short exact sequence of finite abelian groups
\[ \xymatrix@1{ 0 \ar[r] & E(K) / n E(K) \ar[r] & \Sel_n(E) \ar[r] & \mathrm{TS}(K, E)[n] \ar[r] & 0}. \]
Since it is often easier to compute $\Sel_n(E)$ than the group $E(K) / n E(K)$, this provides a useful tool for studying the group of rational points $E(K)$. However, computing $\Sel_n(E)$ for reasonably complicated curves $E$, even when an algorithm is known, can require a large amount of effort. For these reasons, it is of interest to understand the behaviour of the groups $\Sel_n(E)$ on average. Recent years have seen striking progress in problems of this type; for some work of particular relevance to this paper, we refer the reader to any of the papers \cite{deJ02, Bha15, Ngo14}. 

In this paper, we prove new results about the average size of the 2-Selmer group of elliptic curves over global fields of positive characteristic. Such a field is, by definition, the function field $K = \bbF_q(X)$ of an algebraic curve over a finite field. We will consider the universal family of elliptic curves with two marked rational points and calculate the average size of the 2-Selmer groups of the curves in this family satisfying certain conditions. We will accomplish this by relating these 2-Selmer groups to the invariant theory of a representation constructed and studied in \cite{Tho13}, and then counting sections of certain associated vector bundles on $X$. 

In order to state our main theorems precisely, we must introduce some notation. If $E/K$ is an elliptic curve, we can associate its relatively minimal regular model $p_E : \cE \to X$ with identity section $O : X \to \cE$. The isomorphism class of the line bundle $\cL_E = (R^1 p_{E, \ast} \cO_\cE)^{\otimes -1}$ is an invariant of $E$, and there are only finitely many elliptic curves over $K$ up to isomorphism with a given $\cL_E$, this number tending to infinity as $\deg \cL_E \to \infty$.

If $\cL$ is a line bundle on $X$, then we write $\cX_\cL$ for the finite set of isomorphism classes of triples $(E, P, Q)$ as follows:
\begin{enumerate}
\item $E / K$ is an elliptic curve such that $\cL_E \cong \cL^{\otimes 2}$ and the fibres of $p_E$ are all of type $I_0$ or $I_1$.
\item $P, Q \in E(K)$ are distinct non-trivial rational points such that sections $\cO, \cP, \cQ : X \to \cE$ associated to the origin of $E$ and the points $P, Q$, respectively, do not intersect.
\end{enumerate}
Provided that the characteristic of $K$ does not divide 6, an elliptic curve $E$ with two non-trivial marked points can be represented by an equation
\begin{equation}\label{eqn_intro_generalized_weierstrass_equation} Y ( XY + 2 q_4 Z^2) = X^3 + p_2 X^2 Z + p_4 X Z^2 + p_6 Z^3, \end{equation}
which sends the marked points, together with the origin, to the line at infinity. The curves in $\cX_\cL$ are exactly those for which the discriminant $\Delta(p_2, \dots, p_6)$ of this equation vanishes to order at most 1 everywhere, when viewed as a section of $H^0(X, \cL^{\otimes 24})$; see \S \ref{sec_elliptic_curves_with_marked_points} below.

 We can now state our first main theorem.
\begin{theorem}\label{thm_intro_theorem_1}
Suppose that $\cha K > 19$. The limit
\[ \lim_{\deg \cL \to \infty} \sum_{(E, P, Q) \in \cX_{\cL}} \frac{ |\Sel_2(E)| \times |\Aut(E, P, Q)|^{-1} \times |E(K)[2]|^{-1}}{|\cX_{\cL}|} \]
exists and equals 12.
\end{theorem}
\begin{remark}
\begin{enumerate}
\item This result is what one might expect given known results about the 2-Selmer groups of elliptic curves without marked points: for the curves in our family, there is a `trivial subgroup' $A_{(E, P, Q)} \subset \Sel_2(E)$, generated by the classes of the points $P$, $Q$, and which generically has size 4. It follows that the remainder $\Sel_2(E) / A_{(E, P, Q)}$ should have average size 3. 
\item We believe that the weighting of Selmer elements by automorphisms is natural; similarly for the weighting by $K$-rational 2-torsion points (which can be thought of as $K$-rational automorphisms of the trivial 2-covering $[2] : E \to E$). In fact, the contribution of $E(K)[2]$ can be suppressed: for the curves we consider, the groups $E(K)[\text{tors}]$ are trivial (because they inject into the product of fibral component groups; but these component groups are all trivial, by hypothesis). 
\item The restriction on the characteristic arises because we need to apply Jacobson--Morozov style results to the Lie algebra over $\bbF_q$ of type $D_4$, for example in the construction of the Kostant section (see Proposition \ref{prop_existence_of_kostant_section} below). It may be possible to relax this restriction slightly.
\end{enumerate}
\end{remark}
Let $G = (\mathrm{SO}_4 \times \mathrm{SO}_4) / \Delta(\mu_2)$, where $\mathrm{SO}_4$ is the split special orthogonal group over $\bbF_q$, and $\mu_2$ is its centre. A key role in our proof of Theorem \ref{thm_intro_theorem_1} is played by a family of canonically defined invariant maps 
\begin{equation}\label{eqn_intro_invariant_map}
\inv = \inv_{(E, P, Q)} : \Sel_2(E) \to G(K) \backslash G(\bbA_K) / \prod_v G(\cO_{K_v}). 
\end{equation}
In fact, our consideration of these maps leads to the following generalization of Theorem \ref{thm_intro_theorem_1}, which is a kind of equidistribution result:
\begin{theorem}\label{thm_intro_theorem_2}
Suppose that $\cha K > 19$. Let $f : G(K) \backslash G(\bbA_K) / \prod_v G(\cO_{K_v}) \to \bbR$ be a bounded function, and let $\tau_G$ denote the Tamagawa measure on $ G(K) \backslash G(\bbA_K) / \prod_v G(\cO_{K_v})$. Then the limit
\[ \lim_{\deg \cL \to \infty} \sum_{(E, P, Q) \in \cX_{\cL}} \sum_{x \in \Sel_2(E) - A_{(E, P, Q)}} \frac{ f(\inv x) \times |\Aut(E, P, Q)|^{-1} \times |E(K)[2]|^{-1}}{|\cX_{\cL}|} \]
exists and equals $\int_{g \in G(K) \backslash G(\bbA_K) / \prod_v G(\cO_{K_v})} f(g) \,d \tau_G$.
\end{theorem}
 Taking $f = 1$ to be the constant function, we recover Theorem \ref{thm_intro_theorem_1} (after accounting for the average number of elements in the group $A_{(E, P, Q)}$, which is a simple task). In general, Theorem \ref{thm_intro_theorem_2} can be interpreted as saying that the invariants of non-trivial Selmer elements of elliptic curves in $\cX_\cL$ become equidistributed in $G(K) \backslash G(\bbA_K) / \prod_v G(\cO_{K_v})$ as $\deg \cL \to \infty$. It would be very interesting to get a better understanding of this phenomenon, which persists in other situations (for example, in the case of 2-Selmer groups of elliptic curves without marked points, in which case $G$ should be replaced by the group $\PGL_2$). Can one relate Theorem \ref{thm_intro_theorem_2} to existing conjectures about statistics of ranks of 2-Selmer groups, as in \cite{Poo12}? 

The proofs of Theorem \ref{thm_intro_theorem_1} and Theorem \ref{thm_intro_theorem_2} rely on a connection between the universal family of elliptic curves $(E, P, Q)$ with two marked points and a certain representation $(G, V)$ which was analyzed in \cite{Tho13} from the point of view of Vinberg theory, and which is constructed using the adjoint group over $\bbF_q$ of type $D_4$; the link here exists because the family of curves (\ref{eqn_intro_generalized_weierstrass_equation}) is a miniversal deformation of the simple curve singularity of type $D_4$.  This connection reduces the problem of counting elements of Selmer groups to that of counting orbits in certain representations of $V$. Using the map $\inv$ described above, we reduce this to a problem of counting sections of certain vector bundles over $X$. 

An interesting point in our proof is the calculation of the invariants of trivial elements of the 2-Selmer group. We can describe these explicitly using the principal cocharacter of the ambient group $H$ of type $D_4$ (inside which the pair $(G, V)$ is constructed); see Lemma \ref{lem_geography_of_trivial_orbits}. This gives a quantitative version of the intuitive statement that `trivial elements appear far into the cusp of $V$'. 

Aside from the intrinsic interest of results like Theorem \ref{thm_intro_theorem_1}, one of our motivations was to understand how the techniques of Bhargava--Shankar  for counting integral orbits in coregular representations (see e.g.\ \cite{Bha15}) can be transferred to this function field setting. Instead of reduction theory we use the Harder--Narasimhan (or Shatz) stratification of the space $G(K) \backslash G(\bbA_K) / \prod_v G(\cO_{K_v})$ by the canonical reduction of $G$-torsors. After some reinterpretation, we find that the methods of Bhargava--Shankar are still very effective. In particular, the technique of `cutting off the cusp' works in a very similar way (compare e.g.\ \cite[\S 5]{Tho14} and the proof of Theorem \ref{thm_average_of_non_trivial_selmer} below).

We have restricted ourselves to pointed curves $(E, P, Q)$ satisfying conditions 1.\ and 2.\ above, since this simplifies our analysis of the invariant map (\ref{eqn_intro_invariant_map}). From the point of view of the invariant theory of $(G, V)$, it corresponds to restricting to orbits with square-free discriminant $\Delta$. It would be possible to remove this restriction, at the cost of a more detailed analysis of integral orbits; for example, the invariant map would become multi-valued, since the uniqueness of integral representatives (see Theorem \ref{thm_properties_of_integral_square_free_orbits}) does not hold in general. Compare \cite[\S 3.2]{Bha15} for the kinds of problems that arise. 

We now describe the structure of this paper. In \S \ref{sec_elliptic_curves_with_marked_points}, we introduce the universal family of elliptic curves with two marked points, and study their projective embeddings and integral models. In \S \ref{sec_invariant_theory}, we introduce the representation $(G, V)$ and describe its invariant theory. We also introduce the discriminant $\Delta$ and the important notion of trivial orbits in $G(K) \backslash V(K)$; these are the orbits that will eventually correspond to elements of the trivial subgroup $A_{(E, P, Q)}$ of the 2-Selmer group. We also give some useful criteria for elements in $V(K)$ either to have vanishing discriminant, or to lie in a trivial orbit. In \S \ref{sec_semistability_and_integration}, we describe the Harder--Narasimhan stratification of $G(K) \backslash G(\bbA_K) / \prod_v G(\cO_{K_v})$ (at the level of points only) and the relation between summing over strata and integrating over the adelic points of parabolic subgroups of $G$. Finally, in \S \ref{sec_counting_2_selmer}, we describe the relation between the pair $(G, V)$ and the family of curves (\ref{eqn_intro_generalized_weierstrass_equation}), and exploit this to prove our main theorems Theorem \ref{thm_average_of_non_trivial_selmer} and Theorem \ref{thm_equidistribution_of_non_trivial_selmer}.

\subsection{Notation}
In this paper, we will generally use the letter $K$ to denote a global field of positive characteristic, therefore the function field $\bbF_q(X)$ of a smooth, projective, geometrically connected curve $X$ over $\bbF_q$. If $v$ is a place of $K$, then we will write $K_v$ for the completion of $K$ at $v$, $\cO_{K_v}$ for the ring of integers of $v$, and $\varpi_v \in \cO_{K_v}$ for a choice of uniformizer. We will write $\ord_{K_v} : K_v^\times \to \bbZ$ for the corresponding normalized discrete valuation, $k(v) = \cO_{K_v} / (\varpi_v)$ for the residue field, and $q_v = |k(v)|$ for the cardinality of the residue field. We will generally fix a separable closure $K^s / K$ and separable closures $K^s_v / K_v$, together with compatible embeddings $K^s \hookrightarrow K^s_v$. We then define $\Gamma_K = \Gal(K^s / K)$ and $\Gamma_{K_v} = \Gal(K_v^s / K_v)$; there are canonical maps $\Gamma_{K_v} \to \Gamma_K$. We let $\kappa(v)$ denote the residue field of $K_v^s$, which is an algebraic closure of $k(v)$. We write $I_{K_v} \subset \Gamma_{K_v}$ for the inertia group.

We write $\widehat{\cO}_K = \prod_v \cO_{K_v}$ for the maximal compact subring of the adele ring $\bbA_K = \prod_v' K_v$. We will write $| \cdot |_v : K_v^\times \to \bbR_{>0}$ for the valuation satisfying $| \varpi_v | = q_v^{-1}$, and $\| \cdot \| = \prod_v | \cdot |_v : \bbA_K^\times \to \bbR_{>0}$ for the adelic norm; it satisfies the product formula $ \| \gamma \| = 1$ for all $\gamma \in K^\times$. If $Y$ is a integral smooth scheme over $K_v$, and $\omega_Y$ is a non-vanishing differential form of top degree on $Y$, then we write $| \omega_Y |_v$ for the corresponding measure on $Y(K_v)$.

If $S$ is a scheme, a reductive group over $S$ is a smooth group scheme $G \to S$ with geometric fibres which are (connected and) reductive. If $G$ is a group scheme over $S$ which acts on another scheme $X \to S$, then for $x \in X(S)$ we write $Z_G(x)$ for the scheme-theoretic stabilizer of $x$. If $Z \subset X$ is a closed subscheme, then we write $Z_G(Z)$ and $N_G(Z)$ for the scheme-theoretic centralizers and normalizers of $Z$. If $G$ is a reductive group over a field then we write $Z_0(G)$ for the identity component of the centre $Z_G$ of $G$. Lie algebras will be denoted using gothic letters (e.g.\ $\Lie G = \frg$). 

If $G$ is a smooth group scheme over $\bbF_q$, and $K = \bbF_q(X)$, then we write $\mu_G$ for the right-invariant Haar measure on $G(\bbA_K)$ which gives measure 1 to the open compact subgroup $G(\widehat{\cO}_K) \subset G(\bbA_K)$. If $G$ is semisimple, then we will write $\tau_G$ for the Tamagawa measure on $G(\bbA_K)$; these two measures are related by the formula (see \cite{Wei95}):
\[ \tau_G = q^{\dim G(1-g_X)} \left[ \prod_v \int_{G(\cO_{K_v})} \,| \omega_G |_v \right] \mu_G, \]
where $\omega_G$ is a non-vanishing invariant differential form of top degree on $G$ (hence defined over $\bbF_q$) and $g_X$ denotes the genus of $X$.

\section{Elliptic curves with two marked points}\label{sec_elliptic_curves_with_marked_points}

Let $k$ be a field of characteristic not dividing 6. We consider tuples $(E, P, Q)$, where $E$ is an elliptic curve over $k$ (with origin point $O \in E(k)$) and $P, Q \in E(k)$ are distinct, non-trivial marked points. 

Such pointed curves have a distinguished class of plane embeddings which are different to the usual Weierstrass embeddings, being defined by the linear system associated to the degree 3 divisor $O + P + Q$. Indeed, this linear system is very ample, so embeds $E$ into the projective plane $\bbP_k^2$ in such a way that the points $O$, $P$, $Q$ are collinear. If $X, Y, Z$ are the co-ordinates on $\bbP^2_k$ then we can assume, after a projective transformation,  that $O$, $P$, $Q$ are given respectively by $[0 : 1 : 0]$, $[1 : 1 : 0 ]$, and $[-1 : 1 : 0]$. The co-ordinate system is then uniquely determined up to substitutions of the form $X \leadsto  a X + b Z$ and $Y \leadsto a  Y + c Z$ with $a \in k^\times$, $b,c \in k$. It is easy to check that there is a unique such substitution with $a = 1$ leading to an equation of the form 
\begin{equation}\label{eqn_generalized_weierstrass_equation_over_a_field} Y ( XY + 2 q_4 Z^2) = X^3 + p_2 X^2 Z + p_4 X Z^2 + p_6 Z^3. 
\end{equation}
We define the associated polynomial $f(x) = x^4 + p_2 x^2 + p_4 x  + p_6 + q_4^2$, and $\Delta(p_2, p_4, q_4, p_6) = \disc f \in \bbZ[p_2, \dots, p_6]$. The following is elementary:
\begin{lemma}\label{lem_equations_of_pointed_curves}
Let $p_2, p_4,  q_4, p_6 \in k$, and let $E$ be the plane curve over $k$ defined by the equation (\ref{eqn_generalized_weierstrass_equation_over_a_field}). Then $E$ is smooth if and only if $\Delta(p_2, p_4, q_4, p_6) \neq 0$. The assignment $(E, P, Q, t) \mapsto (p_2, p_4,q_4, p_6)$ defines a bijection between the following two sets:
\begin{itemize}
\item The set of tuples $(E, P, Q, t)$, where $E$ is an elliptic curve over $k$ and $P, Q \in E(k)$ are distinct non-trivial rational points, and $t$ is a basis for $H^0(E, \cO_E(O) / \cO_E)$. These tuples are considered up to isomorphism (i.e.\ isomorphisms $\varphi : E \to E'$ of elliptic curves which preserve the other data).
\item The set of tuples $(p_2, p_4, q_4, p_6) \in k^4$ such that $\Delta(p_2, p_4, q_4, p_6) \neq 0$.
\end{itemize}
Under this bijection, a tuple $(E, P, Q, \lambda t)$ ($\lambda \in k^\times$) corresponds to $(\lambda p_2, \lambda^2 p_4, \lambda^2 q_4, \lambda^3 p_6)$.
\end{lemma}
\begin{proof}
The only thing to note is that the bijection is normalized by the requirement that $Y / Z \in H^0(E, \cO_E(O + P + Q))$ has image in $\cO_E(O) / \cO_E$ equal to $t$.
\end{proof}
A similar story works over a more general base:
\begin{proposition}
Let $S$ be a $\bbZ[1/6]$-scheme, and let $p : E \to S$ be a family of elliptic curves equipped with identity section $O \in E(S)$ and sections $P, Q \in E(S)$ such that on every fibre, the associated points are distinct and non-trivial. Let $\cL = (p_\ast [ \cO_E(O) / \cO_E ])^{\otimes -1}$. Then $\cL$ is an invertible $\cO_S$-module, and there are canonically determined sections $p_2 \in H^0(S, \cL)$, $p_4, q_4 \in H^0(S, \cL^{\otimes 2})$, and $p_6 \in H^0(S, \cL^{\otimes 3})$, such that $E$ is isomorphic to the subscheme of $\bbP( \cL \oplus \cL \oplus  \cO_S)$  defined by the equation
\begin{equation}\label{eqn_generalized_weierstrass_equation} Y( XY + 2 q_4 Z^2) = X^3 + p_2 X^2 Z + p_4 X Z^2 + p_6 Z^3, 
\end{equation}
where $(X, Y, Z)$ is the co-ordinate system relative to the decomposition $  \cL \oplus \cL \oplus \cO_S$. Moreover, $\Delta(p_2, \dots p_6) \in H^0(S, \cL^{\otimes 12})$ is an everywhere non-vanishing section.

Conversely, suppose given an invertible $\cO_S$-module $\cL$, together with sections $p_2, \dots, p_6$ as above such that $\Delta(p_2, p_4, q_4, p_6)$ is a non-vanishing section of $\cL^{\otimes 12}$. Then the relative curve defined by the equation (\ref{eqn_generalized_weierstrass_equation}) is an elliptic curve with marked points at infinity that are distinct and non-trivial in each fibre. 
\end{proposition}
We can use this theory to describe integral models of such triples $(E, P, Q)$ over a Dedekind scheme. Let $S$ be a Dedekind scheme on which 6 is a unit, let $K = K(S)$, and let $\cL$ be an invertible $\cO_S$-module. Suppose given sections $p_2 \in H^0(S, \cL)$, $p_4, q_4 \in H^0(S, \cL^{\otimes 2})$, and $p_6 \in H^0(S, \cL^{\otimes 3})$ such that $\Delta(p_2, p_4, q_4, p_6) \in H^0(S, \cL^{12})$ is non-zero. Then the equation (\ref{eqn_generalized_weierstrass_equation}) defines a proper flat morphism $p : \cE \to S$ with smooth generic fibre (and indeed, singular fibres exactly above those points of $S$ where $\Delta$ vanishes).

We call the data of $(\cL, p_2, \dots, p_6)$ minimal if we cannot find an invertible subsheaf $\cM \subset \cL$ such that the sections $p_2, \dots, p_6$ all come from $\cM$. The minimal data is uniquely determined by the triple $(E, P, Q)$ over $K$, in the following sense: if $(\cL, p_2, \dots, p_6)$ and $(\cM, p'_2, \dots, p_6')$ are two sets of minimal data associated to $E$, then we can find an isomorphism $f : \cL \to \cM$ of invertible $\cO_S$-modules such that $f(p_2, \dots, p_6) = (p_2', \dots, p_6')$. Indeed, it follows from Lemma \ref{lem_equations_of_pointed_curves} that we can find an isomorphism $f_\eta : \cL_\eta \to \cM_\eta$ over the generic point $\eta$ of $S$ such that $f(p_2, \dots, p_6) = (p_2', \dots, p_6')$. Choosing an isomorphism $\cL_\eta \cong K$, we see that both $\cL$ and $\cM$ can be characterized as the smallest invertible subsheaves of $K$ containing the sections $p_2, \dots, p_6$ in their respective tensor powers.

We refer to the morphism $p : \cE \to S$ associated to minimal data $(\cL, p_2, \dots, p_6)$ as a minimal integral model of the triple $(E, P, Q)$. By the above discussion, it is also uniquely determined up to isomorphism by $(E, P, Q)$. We can describe this minimal model in elementary terms in case $K = \bbF_q(X)$ is the function field of a smooth, projective, geometrically connected algebraic curve over $\bbF_q$. Let $(E, P, Q)$ be an elliptic curve over $K$ with two distinct non-trivial marked rational points, and choose an arbitrary equation of type (\ref{eqn_generalized_weierstrass_equation}) with $p_2, \dots, p_6 \in K$. Then for each place $v$ of $K$ there is a unique integer $n_v$ satisfying the following conditions:
\begin{enumerate}
\item The tuple $(\varpi_v^{n_v} p_2, \varpi_v^{2 n_v} p_4, \varpi_v^{2 n_v} q_4, \varpi_v^{3 n_v} p_6)$ has co-ordinates in $\cO_{K_v}$.
\item The integer $n_v$ is minimal with respect to this property.
\end{enumerate}
We then define $\cL \subset K$ to be the invertible subsheaf whose sections over a Zariski open $U \subset X$ are given by the formula
\[ \cL(U) = K \cap \left[ \prod_{v \in U} \varpi_v^{-n_v} \cO_{K_v} \right]. \]Then $p_2, \dots, p_6$ are sections of the tensor powers of $\cL$, and the tuple $(\cL, p_2, \dots, p_6)$ is minimal. 

In this paper we will ultimately only be interested in those curves $(E, P, Q)$ for which the associated minimal data $(\cL, p_2, \dots, p_6)$ satisfies the following two conditions:
\begin{enumerate}
\item The line bundle $\cL$ is a square: $\cL \cong \cM^{\otimes 2}$.
\item The discriminant $\Delta(p_2, \dots, p_6) \in H^0(S, \cL^{12}) \cong H^0(S, \cM^{24})$ is square-free.
\end{enumerate}
The reason for this restriction is that these are exactly the curves which are related to orbits of squarefree discriminant in a certain representation, to be considered in the next section. We now give a geometric characterization of curves of square-free discriminant. 
\begin{lemma}\label{lem_characterization_of_square_free_discriminant}
Let $R$ be a DVR in which 6 is a unit, let $K = \Frac R$, and let $S = \Spec R$. Let $(E, P, Q)$ be an elliptic curve over $K$ together with distinct non-trivial marked points $P, Q \in E(K)$. Let $\Delta \in R$ denote the discriminant of a minimal integral model of $(E, P, Q)$ over $S$, therefore determined up to $R^\times$-multiple. Then $\ord_K \Delta \leq 1$ if and only if the following conditions are satisfied:
\begin{enumerate}
\item The minimal regular model of $E$ over $S$ has special fibre of type $I_0$ or $I_1$.
\item The reductions modulo $\ffrm_R$ in the minimal regular model of $E$ of the points $P, Q \in E(K)$ are distinct and non-trivial.
\end{enumerate}
\end{lemma}
\begin{proof}
First let $\cE$ denote a minimal integral model of $E$ over $S$. If $\Delta \in R^\times$, then $\cE \to S$ is smooth, $E$ has good reduction and the points $P, Q$ indeed remain distinct in the special fibre. If the discriminant vanishes to order 1, then the model $\cE$ is regular, with irreducible special fibre. It follows that $\cE$ is the minimal regular model of $E$, which therefore has reduction of type $I_1$.

Now let us assume that $E$ has reduction of type $I_0$ or $I_1$, with the points $O, P, Q$ remaining distinct in the special fibre of the minimal regular model. Let $\cE$ denote the minimal regular model of  $E$, and let $D \subset \cE$ denote the divisor $O + P + Q$ in $\cE$. Fix an isomorphism $H^1(\cE, \cO_E) \cong R$; there is a canonical isomorphism
\[ H^0(\cE, \cO_E(O)|_O) \cong H^1(\cE, \cO_E) \cong R, \]
and similarly with $O$ replaced by $P$ or $Q$. The exact sequence of sheaves
\[ \xymatrix@1{ 0 \ar[r] & \cO_\cE \ar[r] & \cO_\cE(D) \ar[r] & \cO_\cE(D)|_D \ar[r] & 0} \]
gives rise to a long exact sequence
\[ \xymatrix@1{ 0 \ar[r] & R \ar[r] & H^0(\cE, \cO_\cE(D)) \ar[r] & R^3 \ar[r] & R \ar[r] & 0,} \]
where the map $R^3 \to R$ is summing co-ordinates. This sequence of finite free $R$-modules remains exact after applying $-\otimes_R k$, from which we see that each map in the sequence has saturated image. We can therefore choose $x, y \in H^0(\cE, \cO_\cE(D))$ which map to $(0, 1, -1)$ and $(-2, 1, 1)$, respectively, in $R^3$; then the elements $1, x, y \in H^0(\cE, \cO_\cE(D))$ span this free $R$-module, and define a map $\cE \to \bbP^2_R$. The elements 
\[ 1, x, y, x^2, xy, y^2, x^3, x^2 y, xy^2, y^3 \in H^0(\cE, \cO_\cE(3D)) \]
generate this free rank 9 $R$-module, and therefore must satisfy an $R$-linear relation. After dividing out by as many as possible powers of the uniformizer we see that this relation is unique up to multiplication by elements of $R^\times$, and has degree 3 term $ a y(y^2 - x^2)$ for some $a \in R^\times$; after multiplying through, we can assume $a = 1$. We are free to replace $x, y$ by $x + b$, $y + c$ for $b, c \in R$, and there is a unique such transformation which puts our given  relation in the form 
\[ y (xy + 2 q_4) = x^3 + p_2 x^2 + p_4 x + p_6 \]
for some $p_2, p_4, q_4, p_6 \in R$. Let $Z \subset \bbP^2_R$ denote the closed subscheme defined by this equation. Then $Z$ is normal and $R$-flat, and we therefore get a morphism $\cE \to Z$. By Zariski's main theorem, this is in fact an isomorphism and we see that $Z$ is regular, which can happen only if $\Delta(p_2, p_4, q_4, p_6)$ has order of vanishing at most 1. 
\end{proof}
If $D$ is a divisor on $X$, then we will write $\cX_D$ for the set of isomorphism classes of triples $(E, P, Q)$ of elliptic curves over $K$ with two marked points such that the minimal data $(\cL, p_2, \dots, p_6)$ satisfies $\cL \cong \cO_X(2D)$, and the discriminant $\Delta(p_2, \dots, p_6) \in H^0(X, \cL^{\otimes 12}) \cong H^0(X, \cO_X(2D))$ is square-free. Lemma \ref{lem_characterization_of_square_free_discriminant} shows that this is the same as the set $\cX_{\cO_X(D)}$ defined in \S \ref{sec_introduction}.

We also write $B_D = \cO_X(2D) \oplus \cO_X(4D) \oplus \cO_X(4D) \oplus \cO_X(6D)$, a vector bundle over $X$, and write $H^0(X, B_D)^\text{sf} \subset H^0(X, B_D)$ for the set of sections $(p_2, p_4, q_4, p_6) \in H^0(X, B_D)$ for which the discriminant $\Delta(p_2, \dots, p_6) \in H^0(X, \cO_X(24D))$ is square-free. We can summarize the results of this section as follows:
\begin{corollary}\label{cor_moduli_of_pointed_curves}
The assignment $\iota : (p_2, \dots, p_6) \mapsto (E, P, Q)$ which sends sections of $H^0(X, B_D)^\text{sf}$ to the curve given by the equation (\ref{eqn_generalized_weierstrass_equation}) is surjective, each fibre having finite cardinality equal to $| \bbF_q^\times | \cdot | \Aut(E, P, Q) |^{-1}$.
\end{corollary}
\begin{proof}
The only thing left to check is the cardinality of the fibres. Let $\bbF_q^\times$ act on $H^0(X, B_D)$ by the formula $\lambda \cdot (p_2, p_4, q_4, p_6) = (\lambda p_2, \lambda^2 p_4, \lambda^2 q_4, \lambda^3 p_6)$.  Lemma \ref{lem_equations_of_pointed_curves} shows that $\bbF_q^\times$ acts transitively on the fibres of $\iota$, and that the stabilizer of any point is $\Aut(E, P, Q)$. The result follows.
\end{proof}
\section{Invariant theory}\label{sec_invariant_theory}

In this section, we introduce the semisimple group $G$ and its representation $V$, the orbits of which will eventually be interpreted as elements of the 2-Selmer groups of elliptic curves of the type considered in \S \ref{sec_elliptic_curves_with_marked_points}. For the moment, $\bbF_q$ denotes a finite field of characteristic prime to 6; we will soon impose more severe restrictions on the characteristic.

\subsection{Preliminaries}\label{sec_preliminaries}

Let $J$ denote the $4 \times 4$ matrix with $1$'s on the anti-diagonal and $0$'s elsewhere, and define a block matrix
\begin{equation}\label{eqn_matrix_of_duality} \Psi = \left( \begin{array}{cc} J & 0 \\ 0 & J \end{array} \right) \in M_{8 \times 8}(\bbZ). 
\end{equation}
We write $\mathrm{SO}_8$ for the special orthogonal group over $\bbF_q$ defined by $\Psi$, $H = \mathrm{SO}_8 / \mu_2$ for its adjoint group, and $H^\text{sc} = \mathrm{Spin}_8$ for its simply connected double cover.  We write $\frh = \Lie H$. We write $\theta$ for the inner involution of $H$ given by conjugation by the element
\begin{equation}\label{eqn_matrix_of_involution} s = \diag(1, -1, -1, 1, 1, -1, -1, 1). 
\end{equation}
We define $G = (H^\theta)^\circ$ (i.e. the identity component of the $\theta$-fixed subgroup of $H$), and $V = \frh^{d \theta = -1}$. There is an isomorphism $G \cong (\SO_4 \times \SO_4) / \Delta(\mu_2)$, where $\SO_4$ is a split special orthogonal group and  $\Delta(\mu_2)$ is the diagonally embedded centre.

We write $T'$ for the (split) diagonal maximal torus of $\SO_8$; a general element has the form
\[ \diag(a, b, b^{-1}, a^{-1}, c, d, d^{-1}, c^{-1}). \]
We write $T$ for the image of $T'$ in $H$. We observe that $T$ is also a maximal torus of $G$. The group $H^\theta$ is disconnected. Its component group $H^\theta / G$ can be computed as follows: let $W_H = N_H(T) / T$ denote the Weyl group of of $H$, $W = N_G(T) / T$ the Weyl group of $G$. Then the map $Z_{W_H}(s) \to H^\theta / G$ is surjective, with kernel equal to $Z_{W}(s)$ (see \cite[\S 2.2]{Hum95}). A calculation shows that the component group is therefore isomorphic to $\bbZ / 2 \bbZ \times \bbZ / 2 \bbZ$. Explicit representatives can be given by the elements $\sigma, \tau \in W_H$ satisfying
\[ \sigma(a, b, c, d) = (a, b, c^{-1}, d^{-1}), \tau(a, b, c, d) = (b, a, d, c), \]
which generate a subgroup $W_0 \subset W_H$ which projects isomorphically to $H^\theta / G$.

We introduce sets of simple roots as follows. A set $R_H \subset X^\ast(T)$ of simple roots for $H$ consists of the characters
\[ \alpha_1 = a/b, \alpha_2 = b/c, \alpha_3 = c/d, \alpha_4 = cd. \]
We let $\alpha_0 = ab$; it is the highest root of $H$. A set $R \subset X^\ast(T)$ of simple roots for $G$ consists of the characters
\[ a_1 = ac, a_2 = a/c, a_3 = bd, a_4 = b/d. \]
The group $G$ is isogenous to $\SL_2^4$, and the group $W_0 \subset W$ commutes with the action of $W_G$ on $X^\ast(T)$ and leaves invariant the set $\{ a_1, \dots, a_4 \}$. Its action on this set is faithful, and identifies $W_0$ with the Klein 4-group $\{ e, (12)(34), (13)(24), (14)(23) \}$. The characters of $T$ appearing in the representation $V$ are exactly the combinations
\[ \frac{1}{2}( \pm a_1 \pm a_2 \pm a_3 \pm a_4), \]
and can thus be thought of as the vertices of a hypercube. Each weight space is 1-dimensional and we thus have $\dim_{\bbF_q} V = 16$. We write $\Phi_V$ for the set of weights appearing in $V$. Any vector $v \in V$ admits a decomposition $v = \sum_{a \in \Phi_V} v_a$. There is a decomposition $\Phi_V = \Phi_V^+ \sqcup \Phi_V^-$ coming from the decomposition of the roots of $H$ into positive and negative roots. We write $n_1, \dots, n_4$ for the basis of $X_\ast(T)_\bbQ$ dual to $a_1, \dots, a_4$. We define a partial order on $\Phi_V$ by setting $a \geq b$ if $n_i(a) \geq n_i(b)$ for each $i = 1, \dots 4$. 
We label these weights in $\Phi_V$ as follows:

\begin{multicols}{2}

\begin{tabular}{c|cccc}
\# & $2n_1$ & $2n_2$ & $2n_3$ & $2n_4$\\
\hline
 1 & 1 & 1 & 1 & 1 \\
 2 & -1 & 1 & 1 & 1 \\
 3 & 1 & -1 & 1 & 1 \\
 4 & 1 & 1 & -1 & 1 \\
 5 & 1 & 1 & 1 & -1 \\
 6 & -1 & -1 & 1 & 1 \\
 7 & -1 & 1 & -1 & 1 \\
 8 & -1 & 1 & 1 & -1 \\
 9 & 1 & -1 & -1 & 1 \\
 10 & 1 & -1 & 1 & -1 \\
 11 & 1 & 1 & -1 & -1 \\
 12 & -1 & -1 & -1 & 1 \\
 13 & -1 &-1 & 1 & -1 \\
 14 & -1 & 1 & -1 & -1 \\
 15 & 1 & -1 & -1 & -1 \\
 16 & -1 & -1 & -1 & -1 \\
\end{tabular}

\columnbreak

{\centering
\begin{tikzpicture}[scale=2.4]
\tikzstyle{every node}=[circle, minimum size=.5cm]
\node (a) at (0,0) {12};

\node (b) at (0.92388, -0.382683) {16};
\node (c) at (-0.382683,	0.92388) {6};

\node (d) at (0.541196,	0.541196) {13};
\node (e) at (0.92388,	0.382683) {9};

\node (f) at (1.84776,0) {15};
\node (g) at (0.382683,	0.92388) {7};
\node (h) at (1.30656,	0.541196) {14};

\node (i) at (0.541196,	1.30656) {3};
\node (j) at (1.46508,	0.92388) {10};
\node (k) at (0,	1.84776) {2};

\node (l) at (0.92388,	1.46508) {8};
\node (m) at (1.30656,	1.30656) {4};
\node (n) at (2.23044,	0.92388) {11};
\node (o) at (0.92388,	2.23044) {1};

\node (p) at (1.84776,	1.84776) {5};

\draw[thick] (o) -- (k);
\draw[thick] (o) -- (p);
\draw[thick] (o) -- (i);
\draw[thick] (o) -- (m);

\draw[thick] (k) -- (l);
\draw[thick] (k) -- (c);
\draw[thick] (k) -- (g);

\draw[thick] (p) -- (l);
\draw[thick] (p) -- (j);
\draw[thick] (p) -- (n);

\draw[thick] (l) -- (d);
\draw[thick] (l) -- (h);

\draw[thick] (i) -- (c);
\draw[thick] (i) -- (j);
\draw[thick] (i) -- (e);

\draw[thick] (m) -- (g);
\draw[thick] (m) -- (n);
\draw[thick] (m) -- (e);

\draw[thick] (c) -- (a);
\draw[thick] (c) -- (d);

\draw[thick] (g) -- (a);
\draw[thick] (g) -- (h);

\draw[thick] (j) -- (d);
\draw[thick] (j) -- (f);

\draw[thick] (n) -- (h);
\draw[thick] (n) -- (f);

\draw[thick] (d) -- (b);

\draw[thick] (h) -- (b);

\draw[thick] (e) -- (a);
\draw[thick] (e) -- (f);

\draw[thick] (a) -- (b);
\draw[thick] (f) -- (b);
\end{tikzpicture}
}
\end{multicols}
The figure above shows the Hasse diagram of $\Phi_V$ with respect to this partial order. The weight labelled 1 is $\alpha_0$. If $M \subset \Phi_V$ is a subset, we will write $\lambda(M) \subset \Phi_V - M$ for the set of maximal elements of $\Phi_V - M$, i.e.\ the set
\[ \{ a \in \Phi_V - M \mid \forall b \in \Phi_V - M, a \leq b \Rightarrow a  = b \}. \]
It is useful to note that the action of $W_0$ preserves the partial order on $\Phi_V$, and consequently commutes with application of the function $\lambda$.

In the paper \cite{Tho13}, we have summarised part of the invariant theory of the pair $(G, V)$ over a field of characteristic 0; in this case, the most important results were established by Kostant--Rallis \cite{Kos71}. They have been extended to positive characteristic in many cases by Levy \cite{Lev07}. We now discuss this. 
\begin{proposition}\label{prop_invariant_theory_of_theta_group} Let $k / \bbF_q$ be a field, and let $k^s / k$ be a separable closure.
\begin{enumerate}
\item The natural inclusions $\bbF_q[V]^G \to \bbF_q[V]^{H^\theta} \to \bbF_q[\frh]^H$ are isomorphisms, and all of these rings are isomorphic to polynomial algebras over $\bbF_q$ on four homogeneous generators of degrees $2, 4, 4,$ and $6$, respectively. We write $\Delta \in \bbF_q[V]^G$ for the restriction of the standard discriminant polynomial of the Lie algebra $\frh$. It is non-zero.
\item Let $B = \Spec \bbF_q[V]^G$, and let $\pi : V \to B$ denote the natural map. Then $\pi$ has reduced, $G^\theta$-invariant fibres.
\item Let $v \in V_k$. Then $Z_{G_k}(v)$ and $Z_{H_k}(v)$ are smooth over $k$.
\item Let $\frc \subset V_k$ be a subspace. We call $\frc$ a Cartan subspace if there exists a maximal torus $C \subset H_k$ such that $\theta(t) = t^{-1}$ for all $t \in C$ and $\Lie C = \frc$. All such subspaces are conjugate under the action of $G(k^s)$.
\item Let $\frc \subset V_k$ be a Cartan subspace. Then the map $N_{G_k}(\frc) \to W(H_k, \frc) = N_{H_k}(\frc) / Z_{H_k}(\frc)$ is surjective, and the natural restriction map $k[V]^G \to k[\frc]^{ W(H_k, \frc)}$ is an isomorphism.
\item Let $v \in V_k$. Then the following are equivalent:
\begin{enumerate}
\item $v$ is semisimple as an element of $\frh_k$.
\item $G_k \cdot v \subset V_k$ is closed.
\item $v$ is contained in a Cartan subspace of $V_k$.
\end{enumerate}
Any such element is called a semisimple element of $V_k$.
\item Let $v \in V_k$. Then the following are equivalent:
\begin{enumerate}
\item $\dim Z_{H_k}(v) = \dim T$.
\item $\dim Z_{G_k}(v) = 0$.
\end{enumerate}
Any such element is called a regular element of $V_k$. The condition of being regular is open, and we write $V^\text{reg} \subset V$ for the open subscheme of regular elements.
\item Let $b \in B(k)$, and let $V_b = \pi^{-1}(b) \subset V$. Then $V_b(k^s)$ contains regular semisimple elements if and only if $\Delta(b) \neq 0$. In this case, $G(k^s)$ acts transitively on $V_b(k^s)$ and for any $v \in V_b(k^s)$, $\frz_{\frh_k}(v) = \Lie Z_{H_k}(v)$ is the unique Cartan subspace of $V_k$ containing $v$.
\end{enumerate}
\end{proposition}
\begin{proof}
Rather than give detailed references to \cite{Lev07}, we simply refer the reader to the introduction of that paper, which features a thorough summary of the results therein. 
\end{proof}
The group $\bbG_m$ acts on $V$ by scalar multiplication; there is an induced $\bbG_m$-action on the quotient $B$ which makes the morphism $\pi : V \to B$ equivariant. We write $B^\text{rs} \subset B$ for the open subscheme where $\Delta$ is non-zero; by the proposition, $\pi^{-1}(B^\text{rs}) = V^\text{rs}$ is the open subscheme of regular semisimple elements of $V$.
\subsection{Singular and trivial orbits}
Let $k / \bbF_q$ be a field. We are now going to give simple criteria in terms of vanishing of certain matrix entries for elements $v \in V_k$ either to satisfy $\Delta(v) = 0$, or to be trivial in a sense we will soon define. 
\begin{lemma}\label{lem_reducible_orbits}
Let $k / \bbF_q$ be a field, and let $v = \sum_{a \in \Phi_V} v_a \in V_k$.
\begin{enumerate}
\item Let $S \subset \{1, 2, 3, 4 \}$ be a two-element subset, and suppose that $v_a = 0$ if $n_i(a) > 0$ for each $i \in S$. Then $\Delta(v) = 0$.
\item Suppose that $v_a = 0$ if $n_i(a) < 0$ for at most one $i \in \{1, 2, 3, 4 \}$. Then $\Delta(v) = 0$.
\end{enumerate}
\end{lemma}
\begin{proof}
We will use the following criterion: let $\frp \subset \frh$ be a $\theta$-stable parabolic subalgebra which contains $\frt = \Lie T$, and let $v \in \frp_k^{d \theta = -1}$. Then $\Delta(v) = 0$. Indeed, if $\Delta(v) \neq 0$ then $v$ is regular semisimple, hence its centralizer $\frc = \frz_{\frh_k}(v)$ is a Cartan subalgebra of $\frh_k$ which is contained in $V_k$.  We have $\dim_k \frc \leq \dim_k \frz_{\frp_k}(v) \leq \dim_k \frc$, hence $\frc = \frz_{\frp_k}(v)$ and  $\frc \subset \frp_k^{d \theta = -1}$. Let $C \subset H_k$ denote the unique maximal torus with $\Lie C = \frc$. We have $\dim Z_{P_k}(v) \geq \dim C$, hence $Z_{P_k}(v) = C$ is smooth and $C \subset P_k$. There is a unique Levi subgroup $L \subset P_k$ containing $C$, which is necessarily stable under the action of $\theta$. The centre $Z_L$ is contained in $C$, on which $\theta$ acts by $t \mapsto t^{-1}$. On the other hand, $L$ projects isomorphically and $\theta$-equivariantly to the Levi quotient of $P_k$, and $\theta$ acts on the centre of this quotient trivially (because it acts trivially on $T$). This contradiction implies that we must have $\Delta(v) = 0$. 

If $S \subset \Phi_V$ is a subset, we write $V_S \subset V$ for the subspace given by the equations $v_a = 0$ ($a \in S$). The four maximal proper parabolic subalgebras $\frp \subset \frh$ which contain the Borel subalgebra corresponding to the root basis $R_H$ have $\frp^{d \theta = -1} = V_S$ for the following sets of weights:
\begin{equation}\label{eqn_parabolic_reducible_subspaces} \begin{aligned} S = & \{ \frac{1}{2}(a_1 + a_2 \pm a_3 \pm a_4) \},\, \{ \frac{1}{2}(a_1 \pm a_2 \pm a_3 + a_4) \},\, \{ \frac{1}{2}(a_1 \pm a_2 + a_3 \pm a_4) \},
 \text{ and } \\ & \{ \frac{1}{2}(a_1 + a_2 + a_3 + a_4), \frac{1}{2}(-a_1 + a_2 + a_3 + a_4), \frac{1}{2}(a_1 - a_2 + a_3 + a_4),\\& \hspace{2cm} \frac{1}{2}(a_1 + a_2 - a_3 + a_4), \frac{1}{2}(a_1 + a_2 + a_3 - a_4) \}. \end{aligned}
\end{equation}
The last of these gives the subspace appearing in the second part of the lemma. On the other hand, each of the subspaces appearing in the first part of the lemma is $W_0$-conjugate to one of the first three appearing in (\ref{eqn_parabolic_reducible_subspaces}). The action of $W_0$ leaves $\Delta$ invariant, so this implies the first part of the lemma.
\end{proof}
We now introduce the Kostant section. This is a section $\kappa : B \to V$ of the morphism $\pi : V \to B$, and which has image consisting of regular elements of $V$. We will follow Slodowy \cite{Slo80} in constructing $\kappa$ using a fixed choice of regular $\frs\frl_2$-triple and we must therefore impose the restriction that the characteristic of $\bbF_q$ exceeds $4 h - 2$, where $h$ is the Coxeter number of $H$, namely 6. We therefore now make the following assumption, which holds for the remainder of \S \ref{sec_invariant_theory}:
\begin{itemize}
\item The characteristic of $\bbF_q$ is at least $23$. 
\end{itemize}
This being the case, we define 
\[ E = \left(
\begin{array}{cccccccc}
 0 & 1 & 0 & 0 & 0 & 0 & 0 & 0 \\
 0 & 0 & 0 & 0 & 1 & 0 & 0 & 0 \\
 0 & 0 & 0 & -1 & 0 & 0 & 0 & 0 \\
 0 & 0 & 0 & 0 & 0 & 0 & 0 & 0 \\
 0 & 0 & 0 & 0 & 0 & 1 & 1 & 0 \\
 0 & 0 & 0 & 0 & 0 & 0 & 0 & -1 \\
 0 & 0 & 0 & 0 & 0 & 0 & 0 & -1 \\
 0 & 0 & -1 & 0 & 0 & 0 & 0 & 0 \\
\end{array}
\right). \]
and $\check{\rho} : \bbG_m \to T$ by the formula
\[ \check{\rho}(t) = (t^3, t^2, t^{-2}, t^{-3}, t, 1, 1, t^{-1}). \]
(Thus in fact $\check{\rho}$, which is the sum of the fundamental coweights, lifts to $X_\ast(T')$.) We have the formula $\Ad \check{\rho}(t)(E) = t E$, and we can decompose $E = X_{\alpha_1} + X_{\alpha_2} + X_{\alpha_3} + X_{\alpha_4}$ as a sum of $T$-eigenvectors corresponding to the simple roots $R_H$.
\begin{proposition}\label{prop_existence_of_kostant_section}
\begin{enumerate}
\item There exists a unique element $F \in V$ such that $\Ad \check{\rho}(t)(F) = t^{-1} F$ and $[E, F] = d \check{\rho}(2)$.
\item Let $\kappa = E + \frz_\frh(F)$, an affine linear subspace of $\frh$. Then $\kappa \subset V$ and the restriction $\pi|_{\kappa} : \kappa \to B$ is an isomorphism.
\end{enumerate}
\end{proposition}
\begin{proof}
The first part is a standard property of $\frs \frl_2$-triples; we could also exhibit $F$ directly. See for example \cite[III, 4.10]{Spr70}. The second part is \cite[\S 7.4, Corollary 2]{Slo80}. An essential role in the proof is played by the fact that for $t \in \bbG_m$, $v \in \kappa$, we have $t\Ad \check{\rho}(t^{-1})(v) \in \kappa$, and this $\bbG_m$-action contracts to the central point $E \in \kappa$. The morphism $\pi|_{\kappa}$ is also clearly equivariant with respect to this $\bbG_m$-action. These properties of the Kostant section will appear again in \S \ref{sec_global_integral_orbits} below.
\end{proof}
\begin{corollary}\label{cor_parameterization_of_orbits}
Let $k / \bbF_q$ be a field, and let $b \in B(k)$, and suppose that $\Delta(b) \neq 0$. Then there is a canonical bijection
\[ G(k) \backslash V(k) \cong \ker(H^1(k, Z_G(\kappa_b)) \to H^1(k, G)). \]
\end{corollary}
\begin{proof}
This follows because $V_b(k^s)$ is a single $G(k^s)$-orbit, and because of the existence of the marked base point $\kappa_b \in V_b(k)$.
\end{proof}
In the situation of the corollary, we refer to the $G(k)$-orbits of the elements $w \cdot \kappa_b$ ($w \in W_0$) as the trivial orbits. We call elements of $V_k = V(k)$ which lie in a trivial orbit trivial elements. Note that this notion depends on $k$ (and indeed, all regular semisimple elements in $V(k^s)$ are trivial over $k^s$).
\begin{lemma}\label{lem_trivial_orbits}
Let $k / \bbF_q$ be a field, and let $v = \sum_{a \in \Phi_V} v_a \in V_k$. Suppose that $v_a = 0$ for all $a \in S$ and $v_a \neq 0$ for all $a \in \lambda(S)$, where $S$ is one of the following sets:
\begin{equation}\label{eqn_weierstrass_vanishing_weights}\begin{aligned}  \{ a_1 + a_2 + a_3 + a_4,  a_1 - a_2 + a_3 + a_4,a_1 + a_2 - a_3 + a_4, a_1 + a_2 + a_3 - a_4 \}, \\ 
 \{ a_1 + a_2 + a_3 + a_4, - a_1 + a_2 + a_3 + a_4, a_1 - a_2 + a_3 + a_4, a_1 + a_2 + a_3 - a_4 \}, \\
  \{a_1 + a_2 + a_3 + a_4, - a_1 + a_2 + a_3 + a_4, a_1 + a_2 - a_3 + a_4, a_1 + a_2 + a_3 - a_4 \}, \\
   \{ a_1 + a_2 + a_3 + a_4, - a_1 + a_2 + a_3 + a_4,a_1 - a_2 + a_3 + a_4, a_1 + a_2 - a_3 + a_4 \}. 
   \end{aligned}
\end{equation}
Then if $\Delta(v) \neq 0$ then $v$ belongs to a trivial orbit of $G(k)$. 
\end{lemma}
\begin{proof}
These sets $S$ form a single $W_0$-orbit, so it suffices to treat one of them, say
\[ S = \{ a_1 + a_2 + a_3 + a_4, a_1 - a_2 + a_3 + a_4, a_1 + a_2 - a_3 + a_4, a_1 + a_2 + a_3 - a_4 \}. \]
In this case, we can compute
\[ \lambda(S) = \{ -a_1 + a_2 + a_3 + a_4, a_1 + a_2 - a_3 - a_4, a_1 -a_2 - a_3 + a_4, a_1 - a_2 + a_3 - a_4 \} = \{ \alpha_1, \alpha_2, \alpha_3, \alpha_4 \}. \]
Thus if $v \in V(k)$ is as in the statement of the lemma, we can write
\[ v = \sum_{i=1}^4 \lambda_i X_{\alpha_i} + \sum_{a \in \Phi_V^-} v_a, \]
where each $\lambda_i \in k^\times$. Since the group $H$ is adjoint, we can find $t \in T(k)$ such that $\alpha_i(t) = \lambda_i$ for each $i = 1, \dots, 4$. Replacing $v$ by $t^{-1} \cdot v$, we can assume that $\lambda_i = 1$ for each $i$.

We claim that this implies that $v$ is $(U^-)^\theta(k)$-conjugate to $\kappa(k)$, where $U^- \subset H$ is the unipotent radical of the Borel subgroup $B^- \subset G$ corresponding to the set $- R_H \subset \Phi(H, T)$ of simple roots. One can show that the natural product map $U^- \times \kappa \to E + \Lie U^- \subset \frh$ is an isomorphism. (The analogous fact in characteristic 0 is employed for a very similar purpose in the proof of \cite[Lemma 2.6]{Tho14}; one can easily check that it is true here as well, under our restrictions on the characteristic.) Since $v$ lies in $E + \Lie U^-_k$, we find that there is a unique pair $(u, b) \in U^-(k) \times \kappa(k)$ such that $u \cdot b = v$, and then $u$ necessarily satisfies $\theta(u) = u$, hence $u \in G(k)$, as required.
\end{proof}
\begin{corollary}\label{cor_identification_of_reducible_or_weierstrass_orbits}
Let $k / \bbF_q$ be a field, and let $v  = \sum_{a \in \Phi_V} v_a \in V_k$. Suppose that $v_a = 0$ for all $a \in S$, where $S$ is one of the following subsets (labelling as in the figure preceding Proposition \ref{prop_invariant_theory_of_theta_group}):
\begin{equation}\label{eqn_parabolic_reducibility_sets} \{ 1, 2, 3, 4, 5\}, \{1, 4, 5, 11\}, \{1, 3, 4, 9\}, \{1, 3, 5, 10\}, \{1, 3, 
4, 5\}, \{1, 2, 3, 5\}, \{1, 2, 4, 5\}, 
\end{equation}
\begin{equation}\label{eqn_weierstrass_reducibility_sets} \{1, 2, 3, 4\}, \{1, 2, 3, 6\}, \{1, 2, 
4, 7\}, \{1, 2, 5, 8\}. 
\end{equation}
 Then either $\Delta(v) = 0$ or $v$ belongs to a trivial orbit of $G(k)$.
\end{corollary}
\begin{proof}
This follows from combining Lemma \ref{lem_reducible_orbits} and Lemma \ref{lem_trivial_orbits}, as we now show. Let $v \in V(k)$. The sets $S$ appearing in (\ref{eqn_parabolic_reducibility_sets}) are exactly those appearing in the statement of Lemma \ref{lem_reducible_orbits}, so the result follows immediately in this case (and indeed we have $\Delta(v) = 0$). The sets $S$ appearing in (\ref{eqn_weierstrass_reducibility_sets}) are exactly those appearing in the statement of Lemma \ref{lem_trivial_orbits}. If $S$ is one of these and $v_a = 0$ for all $a \in S$, then there are two possibilities: either $v_a \neq 0$ for all $a \in \lambda(S)$, or there exists $b \in \lambda(S)$ such that $v_a = 0$ for all $a \in S' = S \cup \{ b \}$. In the first case, Lemma \ref{lem_trivial_orbits} shows that $\Delta(v) = 0$ or $v$ belongs to a trivial orbit. In the second case, we see by inspection that $S'$ is one of the sets appearing in (\ref{eqn_parabolic_reducibility_sets}), hence $\Delta(v) = 0$.
\end{proof}

\section{Interlude on $G$-bundles, semi-stability, and integration}\label{sec_semistability_and_integration}

In this section, we review the parameterization of $G$-torsors on curves by adeles and its relation to integration. We also recall the theory of Harder--Narasimhan filtrations and canonical reductions for $G$-torsors, which will be our substitute for reduction theory when it comes to counting points later on. Let $\bbF_q$ be a finite field.

Let $M$ be a smooth affine group scheme over $\bbF_q$. By definition, an $M$-torsor over a scheme $S / \bbF_q$ is a scheme $F \to S$, equipped with a right action of $M_S$, and locally on $S$ (in the \'etale topology) isomorphic to the trivial torsor $M_S$. A morphism $F \to F'$ of $M$-torsors over $S$ is a morphism $F \to F'$ respecting the $M$-action. A torsor $F \to S$ is trivial (i.e.\ isomorphic to the trivial torsor $M_S$) if and only if it admits a section. The set of isomorphism classes of torsors over $S$ is in bijection with $H^1(S, M)$ (non-abelian \'etale cohomology). 

If $M' \subset M$ is a closed subgroup, still smooth over $\bbF_q$, then a reduction of $F \to S$ to $M'$ is a pair $(F', \varphi)$, where $F' \to S$ is an $M'$-torsor and $\varphi :  F' \times_{M'} M \to F$ is an isomorphism. Giving a reduction of $F$ to $M'$ is then equivalent to giving a section of the sheaf quotient $ F / M'$.

Let $X$ be a smooth, projective, geometrically connected curve over $\bbF_q$, and let $K = \bbF_q(X)$. Suppose that $M$ is connected. We say that an $M$-torsor $F \to X$ is \emph{rationally trivial} if $F_K = F \times_X \Spec K$ is a trivial $M$-torsor. This will always be the case if $M$ satisfies the Hasse principle over $K$. Indeed, each pointed set $H^1(\cO_{K_v}, M)$ is trivial (by Lang's theorem and Hensel's lemma). It is useful to note that if $M$ is split reductive, and $P \subset M$ is a parabolic subgroup, then for any rationally trivial $M$-torsor $F \to X$ with a reduction $F_P \to X$ to $P$, $F_P$ is also rationally trivial. Indeed, the morphism $M \to  M /P$ admits Zariski local sections, and $ F_P / P$ defines a $K$-point of $F / P$.

For any connected smooth affine group $M$, the rationally trivial torsors over $X$ can be parameterized using adeles. Indeed, if $\cY_M$ denotes the set of isomorphism classes of such torsors, then there is a canonical bijection
\begin{equation}\label{eqn_adelic_parameterization_of_torsors} \cY_M \cong M(K) \backslash M(\bbA_K) / M(\widehat{\cO}_K). 
\end{equation}
This is a consequence of fpqc descent; see \cite[Appendix]{Gil02}. We can describe the bijection explicitly as follows: given such a torsor $F \to X$, choose sections $x_0 \in F(K)$, $x_v \in F(\cO_{K_v})$ for each place $v$. Then for each $v$ there is a unique element $m_v \in M(K_v)$ such that $x_0 m_v = x_v$, and we assign to $F$ the element $m_F = (m_v)_v \in M(\bbA_K)$. The class $[(m_v)_v] \in M(K) \backslash M(\bbA_K) / M(\widehat{\cO}_K)$ is then clearly well-defined. If $m \in M(\bbA_K)$, we will write $F_m$ for the corresponding $M$-torsor over $X$. We can describe the group of automorphisms of $F_m \to X$ in these terms: we have an isomorphism $\Aut(F_m) \cong M(K) \cap m M(\widehat{\cO}_K) m^{-1}$. It follows that the correspondence (\ref{eqn_adelic_parameterization_of_torsors}) can instead be thought of as an equivalence of groupoids. 

We will henceforth identify $\cY_M$ with this adelic double quotient. We endow $\cY_M$ with its counting measure $\nu_M$, each point $F \in \cY_M$ being weighted by $| \Aut(F) |^{-1}$. If $\mu_M$ is the (right-invariant) Haar measure on $M(\bbA_K)$ which gives $M(\widehat{\cO}_K)$ volume 1, and with modulus $\Delta_l : M(\bbA_K) \to \bbR_{>0}$ defined by the formula
\[ \int_{m' \in M(\bbA_K)} f(m^{-1} m') \, d \mu_M = \Delta_l(m) \int_{m' \in M(\bbA_K)} f(m') \, d \mu_M, \]
 then we have the formula for any compactly supported function $f : \cY_M \to \bbR$:
\begin{equation}\label{eqn_counting_measure_and_adele_measure} \int_{F \in \cY_M} f(F) \, d \nu_M = \int_{m \in M(K) \backslash M(\bbA_K)} f(F_m) \Delta_l(m)^{-1} \, d \mu_M. \end{equation}
An important special case arises when $M$ is a split reductive group and $P \subset M$ is a parabolic subgroup with Levi decomposition $P = L_P N_P$. In this case we define a character $\delta_P \in X^\ast(P)$ by $\delta_P(p) = \det \Ad(p)|_{\Lie N_P}$. A right-invariant Haar measure is given by the formula
\begin{equation} \int_{p \in P(\bbA_K)} f(p) \, d \mu_P = \int_{l \in L_P(\bbA_K)} \int_{n \in N_P(\bbA_K)} f(nl) \, d\mu_{N_P} \, d\mu_{L_P}, 
\end{equation}
and the modulus character of $P$ is $\Delta_l(p) = \| \delta_P(p) \|$, where $\| \cdot \|$ is the adele norm. In this case (\ref{eqn_counting_measure_and_adele_measure}) becomes
\begin{equation}
\label{eqn_counting_measure_and_adele_measure_for_parabolics} \int_{F \in \cY_P} f(F) \, d \nu_P = \int_{p \in P(K) \backslash P(\bbA_K)} f(F_p) \| \delta_P(p) \|^{-1} \, d \mu_P.
\end{equation}
Now suppose that $G$ is a reductive group over $\bbF_q$ with split maximal torus and Borel subgroup $T \subset B \subset G$. Let $P \subset G$ be a standard parabolic subgroup, i.e.\ one containing $B$, and let $P = L_P N_P$ be its standard Levi decomposition; thus $L_P$ is the unique Levi subgroup of $P$ containing $T$. If $F_P \to X$ is a $P$-torsor, we can associate to it an element $\sigma_{F_P} \in X_\ast(Z_0(L_P))_\bbQ \subset X_\ast(T)$, uniquely characterized by the requirement that for any $\chi \in X^\ast(P)$, the line bundle $\cL_\chi = F_P \times_{P, \chi} \bbA^1_{\bbF_q}$ has degree $\deg \cL_\chi = \langle  \sigma_{F_P}, \chi\rangle$. 

We call $\sigma_{F_P}$ the slope of $F_P$. If $\sigma, \tau \in X_\ast(T)_\bbQ$, then we write $\sigma \leq \tau$ if $\langle \tau - \sigma, \alpha \rangle \geq 0$ for all $B$-positive roots $\alpha \in \Phi(G, T)$. The following formulations are taken from \cite{Sch15}.
\begin{definition} Let $G$ be a split reductive group over $\bbF_q$, with split maximal torus and Borel subgroup $T \subset B \subset G$. Let $R \subset \Phi(G, T)$ denote the set of simple roots corresponding to $B$. Let $F \to X$ be an $G$-torsor.
\begin{enumerate}
\item We say that $F$ is semi-stable if for any standard parabolic subgroup $P \subset G$ and any reduction $F_P \to X$ of $F$, we have $\sigma_{F_P} \leq \sigma_{F}$.
\item Let $P$ be a standard parabolic subgroup with Levi quotient $L_P$, and let $F_P \to X$ be a reduction of $F$ to $P$.  We say that $F_P$ is canonical if $F_P \times_P L_P$ is semi-stable and if for any simple root $\alpha \in R - \Phi(L_P, T)$, we have $\langle \sigma_{F_P}, \alpha \rangle > 0$.
\end{enumerate}
\end{definition}
The following result justifies the use of the word `canonical':
\begin{theorem}\label{thm_existence_of_canonical_reduction}
Let $F \to X$ be a $G$-torsor. Then there exists exactly one pair $(P, F_P)$ consisting of a standard parabolic subgroup $P \subset G$ and a reduction $F_P \to X$ of $F$ which is canonical. 
\end{theorem}
\begin{proof}
See \cite[Theorem 2.1]{Sch15} and the remarks following.
\end{proof}
This theorem allows us to decompose $\cY_G = \sqcup_P \cY_{G, P}$, where $\cY_{G, P}$ denotes the set of $G$-torsors on $X$ which admit a canonical reduction to the standard parabolic subgroup $P$. We then have an identification
\begin{equation}\label{eqn_identification_of_P_semistable_points} \cY_{G, P} = P(K) \backslash P(\bbA_K)^{\text{pos, ss}} / P(\widehat{\cO}_K), 
\end{equation}
where we define
\[ P(\bbA_K)^{\text{pos}} = \{ p \in P(\bbA_K) \mid \forall \alpha \in R - \Phi(L_P, T), \langle m_P(p), \alpha \rangle >0  \}, \]
\[ P(\bbA_K)^{\text{ss}} = \{ p \in P(\bbA_K) \mid F_p \times_P L_P \text{ semi-stable} \}, \]
and
\[ P(\bbA_K)^{\text{pos, ss}} = P(\bbA_K)^{\text{pos}} \cap P(\bbA_K)^{\text{ss}}. \]
Here we write
\[ m_P : P(\bbA_K) \to \Hom(X^\ast(L_P), \bbQ) \cong X_\ast(Z_0(L_P))_\bbQ \subset X_\ast(T)_\bbQ, p \mapsto (\chi \mapsto \log_q \| \chi(p) \|). \]
We observe the formulae
\begin{equation}\label{eqn_modulus_character_of_parabolic}
m_P(p) = \sigma_{F_p} \text{ and }\Delta_l(p) = \| \delta_P(p) \| = q^{\langle m_P(p), \delta_P \rangle}. 
\end{equation}
We define $\Lambda_P^{\text{pos}} = m_P(P(\bbA_K)^\text{pos}) \subset X_\ast(T)_\bbQ$. Theorem \ref{thm_existence_of_canonical_reduction} implies that (\ref{eqn_identification_of_P_semistable_points}) is an isomorphism of groupoids: if $p \in P(\bbA_K)^{\text{pos, ss}}$, then the inclusion $P(K) \cap p P(\widehat{\cO}_K) p^{-1} \to G(K) \cap p G(\widehat{\cO}_K) p^{-1}$ is an isomorphism (because any automorphism of a $G$-torsor must preserve its canonical reduction). This leads to the following lemma.
\begin{lemma}\label{lem_integration_of_boundary_terms}
There exists a constant $C > 0 $ depending only on $X$ such that for any standard parabolic subgroup $P \subset G$ and function $f : X_\ast(Z_0(L_P))_\bbQ \to \bbR_{\geq 0}$, we have 
\[ \int_{F \in \cY_{G, P}} f(\sigma_{F_P}) \, d \nu_G \leq C \sum_{\sigma \in \Lambda_P^\text{pos}} q^{-\langle \sigma, \delta_P \rangle} f(\sigma).\]
\end{lemma}
\begin{proof}
Let $P(\bbA_K)^0 = \ker m_P$. Then $P(K) \subset P(\bbA_K)^0$ and the quotient $P(K) \backslash P(\bbA_K)^0$ has finite $\mu_P$-volume. We choose the constant $C$ to exceed the volume of $P(K) \backslash P(\bbA_K)^0$ for all standard parabolic subgroups of $G$. Then (\ref{eqn_counting_measure_and_adele_measure_for_parabolics}) and (\ref{eqn_modulus_character_of_parabolic}) give
\[ \int_{F \in \cY_{G, P}} f(\sigma_{F_P}) \, d \nu_P \leq \int_{p \in P(K) \backslash P(\bbA_K)^\text{pos}} f( m_P(p) ) \| \delta_P(p) \|^{-1} \, d \mu_P  \leq C \sum_{\sigma \in \Lambda_P^\text{pos}} q^{-\langle \sigma, \delta_P \rangle} f(\sigma), \]
as required.
\end{proof}
We need to discuss the behaviour of the canonical reduction under certain functorialities. For this it is useful to recall that giving a $\GL_n$-torsor over $X$ is equivalent to giving a vector bundle over $X$ of rank $n$, via $F \mapsto F \times_{\GL_n} \bbA^n_{\bbF_q}$. If $\cE \to X$ is a vector bundle, then its slope is defined to be $\mu(\cE) = \deg \cE / \rank \cE$. A vector bundle is said to be semi-stable if for any vector subbundle $\cF \subset \cE$, we have $\mu(\cF) \leq \mu(\cE)$. This is equivalent to the semi-stability of the corresponding $\GL_n$-torsor, and Theorem \ref{thm_existence_of_canonical_reduction} is equivalent to the following statement: given a vector bundle $\cE \to X$ of rank $n$, there is a unique filtration 
\begin{equation}\label{eqn_harder_narasimhan_filtration} 0 \subset \cE_1 \subset \cE_2 \subset \dots \subset \cE_m = \cE 
\end{equation}
by vector subbundles such that each subquotient $\cE_{i+1} / \cE_i$ is (non-zero and) semi-stable, and we have the chain of inequalities
\begin{equation} \mu(\cE_1) > \mu(\cE_2 / \cE_1) > \dots > \mu(\cE_m / \cE_{m-1}). 
\end{equation}
This is the Harder--Narasimhan filtration of $\cE$. It will play a key role for us because of the following lemma.
\begin{lemma}\label{lem_cliffords_theorem}
Let $\cE$ be a semi-stable vector bundle over $X$ of rank $n$. Let $g_X$ denote the genus of $X$.
\begin{enumerate}
\item If $\mu(\cE) < 0$, then $h^0(X, \cE) = 0$.
\item If $0 \leq \mu(\cE) \leq 2g_X - 2$, then $h^0(X, \cE) \leq n(1 + \mu(\cE) / 2)$.
\item If $\mu(\cE) > 2g_X - 2$, then $h^0(X, \cE) = n(1 - g_X + \mu(\cE))$.
\end{enumerate}
\end{lemma}
\begin{proof}
The first and third points are well-known properties of semi-stable bundles and follow easily from the definition, together with the Riemann--Roch theorem. The second point is a generalization of Clifford's theorem for line bundles, see \cite[Theorem 2.1]{Bra97}.
\end{proof}
\begin{corollary}\label{cor_vanishing_sections_of_lowest_slope_part}
Let $\cE \to X$ be a vector bundle of rank $n$ and slope $\mu(\cE) = 0$, and let its Harder--Narasimhan filtration be as in (\ref{eqn_harder_narasimhan_filtration}). Let $0 \leq k \leq m+1$ be such that we have
\[ \mu(\cE_1) > \mu(\cE_2 / \cE_1) > \dots > \mu(\cE_k / \cE_{k-1}) > 0 > \mu(\cE_{k+1} / \cE_k) > \dots > \mu(\cE_m / \cE_{m-1}), \]
and let $q_0 = \mu(\cE_m / \cE_{m-1})$. Let $D$ be a divisor on $X$ such that $\deg D > 0$.
\begin{enumerate}
\item If $\deg D + q_0 < 0$, $h^0(X, (\cE_m / \cE_{k})(D)) = 0$ and $h^0(X, \cE(D)) \leq n(1 + \deg D) - (\rank \cE_m / \cE_{k}) \cdot (1 + \mu(\cE_m / \cE_k) + \deg D)$.
\item If $\deg D + q_0 > 2g_X - 2$, then $h^0(X, \cE(D)) = n(1 - g_X +\deg D)$.
\item If $0 \leq \deg D + q_0 \leq 2g_X - 2$, then $h^0(X, \cE(D)) \leq n(1 +\deg D)$.
\end{enumerate}
\end{corollary}
\begin{proof}
We prove the second part first. There are exact sequences for each $i \geq 1$:
\[ \xymatrix@1{ 0 \ar[r] & \cE_{m-i} / \cE_{m-(i+1)} \ar[r] & \cE_m / \cE_{m-(i+1)} \ar[r] & \cE_m / \cE_{m-i} \ar[r] & 0.} \]
We have $\mu(\cE_{m-i} / \cE_{m-(i+1)}) > 2g_X - 2$ for each $i \geq 1$, hence $h^1(\cE_{m-i} / \cE_{m-(i+1)}) = 0$. It follows that $h^0(\cE) = \sum_{i \geq 0} h^0(\cE_{m-i} / \cE_{m-(i+1)}) = n(1 - g_X + \mu(\cE(D)) ) = n(1 - g_X + \deg D)$. The first and third parts can be proved using the same exact sequences, except that we no longer need to calculate any $H^1$ (since we are only looking for upper bounds).
\end{proof}
Consider again a reductive group $G$ over $\bbF_q$ with split maximal torus and Borel subgroup $T \subset B \subset G$. Let $V$ be a finite-dimensional representation of $G$. If $F \to X$ is a $G$-torsor, then $\cV = F \times_G V$ is a vector bundle over $X$. If $F = F_g$ for some $g \in G(\bbA_K)$, then we write $\cV_g = F_g \times_G V$. For any Zariski open subset $U \subset X$, we can identify 
\begin{equation}\label{eqn_sections_of_bundle_associated_to_torsor} H^0(U, \cV_g) = V(K) \cap \prod_{v \in U} g_v V(\cO_{K_v}). 
\end{equation}
If $V$ has `small height', then we can describe the  Harder--Narasimhan filtration of $\cV$ explicitly. Let $F_P \to X$ denote the canonical reduction of $F$. For each rational number $q$, we define
\begin{equation} V_q = \bigoplus_{\substack{\lambda \in X^\ast(T) \\ \langle \sigma_{F_P}, \lambda \rangle \geq q}} V_\lambda \subset V, 
\end{equation}
$V_\lambda \subset V$ denoting the $\lambda$-weight space. This defines a decreasing filtration $V_\bullet$ of $V$. The subspaces are $P$-invariant, and the action of $P$ on the graded pieces factors through the Levi quotient $L_P$ (see \cite[Lemma 5.1]{Sch15}). By pushout, we get a filtration $\cV_\bullet = F_P \times_P V_q$ of $\cV$ by subbundles indexed by rational numbers $q$. We then have the following result.
\begin{theorem}\label{thm_filtration_of_associated_bundle} Let $V$ be a finite-dimensional representation of $G$, and let $\check{\rho} \in X_\ast(T)_\bbQ$ denote the sum of the fundamental coweights. Suppose that for all weights $\lambda \in X^\ast(T)$ such that $V_\lambda \neq 0$, we have $2 \langle \check{\rho}, \lambda \rangle < \cha \bbF_q$. (This condition depends only on the pair $(G, V)$ and not on the choice of $T$ or $B$.) Then:
\begin{enumerate} \item Each associated bundle $\gr_q \cV_\bullet \cong F_P \times_P \gr_q V_\bullet$ is (either non-zero or) semi-stable of slope $q$.
\item The subbundles $\cV_q = F_P \times_P V_q$ of $\cV$ are the constituents of the Harder--Narasimhan filtration of $\cV = F \times_G V$.
\end{enumerate}
\end{theorem}
\begin{proof}
The calculation of \cite[Proposition 5.1]{Sch15} goes over verbatim to show that the associated bundles of the graded pieces have the claimed slopes. What we need to justify here is that they are semi-stable. In \emph{loc. cit.} this is justified by appeal to the results of \cite{Ram84}, which apply when the ground field has characteristic 0. In the present case we can appeal instead to the main theorem of \cite{Ila03}, which is extended to reductive groups $G$ as \cite[Proposition 4.9]{Bis04}.
\end{proof}

We conclude this section by applying the preceding results to the pair $(G, V)$ constructed in \S \ref{sec_invariant_theory}. We therefore assume now that $\cha \bbF_q > 3$. We recall that $G$ has the root basis $R = \{ a_1, a_2, a_3, a_4 \}$. We write $R^- = - R$ for the negative of this root basis, and $B \subset G$ for the Borel subgroup corresponding to $R^-$. We call a parabolic subgroup $P \subset G$ containing $B$ a standard parabolic; any such parabolic has a canonical Levi decomposition $P = L_P N_P$, where $L_P$ is the unique Levi subgroup of $P$ which contains the maximal torus $T$. 

If $P \subset G$ is a standard parabolic subgroup, and $D$ is a divisor on $X$, then we define a further decomposition of $\cY_{G, P} \subset \cY_G$ as follows:
\[ \cY_{G, P} = \cY_{G, P}(D)^{<0} \sqcup \cY_{G, P}(D)^{\text{sp}} \sqcup \cY_{G, P}(D)^{>2g_X-2}. \]
where $\cY_{G, P}(D)^{<0}$ denotes the set of $G$-torsors $F \to X$ for which the lowest slope piece of the Harder--Narasimhan filtration of $F \times_G V$ has slope $q_0$ satisfying $\deg D + q_0 < 0$; $\cY_{G, P}(D)^{\text{sp}}$ the set for which $0 \leq \deg D + q_0 \leq 2g - 2$; and $\cY_{G, P}(D)^{>2g_X-2}$ the set for which $\deg D + q_0 > 2g_X - 2$. 
We can reformulate Corollary \ref{cor_vanishing_sections_of_lowest_slope_part} as follows:
\begin{corollary}\label{cor_reducible_orbits_in_canonical_reduction} Let $g = [(g_v)_v] \in \cY_{G, P}$, and let $F_P \to X$ denote the canonical reduction of $F_g$. Suppose that $\deg D > 0$, and let $M \subset \Phi_V$ denote the set of weights $a \in \Phi_V$ such that $\langle \sigma_{F_P}, a \rangle + \deg D < 0$. Then:
\begin{enumerate}
\item If $g \in \cY_{G, P}(D)^{<0}$ (i.e.\ $M$ is non-empty), then $|H^0(X, \cV_g(D))| \leq q^{\dim V(1 + \deg D) - |M|(1 + \deg D + \sum_{a \in M} \langle \sigma_{F_P}, a \rangle )}$. 
\item If $g \in \cY_{G, P}(D)^{\text{sp}}$, then $| H^0(X, \cV_g(D)) | \leq q^{\dim V(1 + \deg D)}$.
\item If $g \in \cY_{G, P}(D)^{>2g_X - 2}$, then $| H^0(X, \cV_g(D)) |= q^{\dim V (1 - g_X + \deg D)}$.
\end{enumerate}
\end{corollary}
We can combine these ideas with Lemma \ref{lem_reducible_orbits} to obtain the following useful principle:
\begin{corollary}\label{cor_parabolic_reducibility_from_canonical_reduction}
Let $P \subset G$ be a standard parabolic subgroup, and suppose that $\dim Z_0(L_P) \leq 2$. Let $D$ be a divisor on $X$, and let $g \in \cY_{G, P}(D)^{<0}$. Then for all $v \in H^0(X, \cV_g(D))$, we have $\Delta(v) = 0$ (as a section of $H^0(X, \cO_X(24D))$).
\end{corollary}
\begin{proof}
Let $D = \sum_v n_v \cdot v$. If $P \subset G$ satisfies $\dim Z_0(L_P) \leq 2$, then the lowest slope piece of the Harder--Narasimhan filtration of $\cV_g$ has dimension at least 4. (It's helpful to recall here that $G$ is isogenous to $\SL_2^4$, and $V$ is then identified with the tensor product of the four 2-dimensional standard representations.) Under the identification
\[ H^0(X, \cV_g(D)) = V(K) \cap \prod_v \varpi_v^{-n_v} g_v V(\cO_{K_v}) \subset V(K), \]
we see that any $v \in H^0(X, \cV_g(D))$ must satisfy the condition of the first part of Lemma   \ref{lem_reducible_orbits}, and therefore satisfy $\Delta(v) = 0$.
\end{proof}
\section{Counting 2-Selmer elements}\label{sec_counting_2_selmer}

In this section, we describe the relation between the representation $(G, V)$ of \S \ref{sec_invariant_theory} and the family of pointed elliptic curves $(E, P, Q)$ described in \S \ref{sec_elliptic_curves_with_marked_points}. We proceed from the rational theory, to the integral theory, and finally combine this with the other results established so far to prove our main theorems (Theorem \ref{thm_average_of_non_trivial_selmer} and Theorem \ref{thm_equidistribution_of_non_trivial_selmer} below). 

We assume throughout \S \ref{sec_counting_2_selmer} that $\bbF_q$ is a finite field of characteristic $\geq 19$, and let $(G, V)$ denote the representation considered in \S \ref{sec_invariant_theory}.

\subsection{$(G, V)$ and 2-descent}

\begin{theorem}
We can find homogeneous generators $p_2, p_4, q_4, p_6 \in \bbF_q[V]^G$ (of degrees 2, 4, 4, and 6, respectively) and a 5-dimensional affine linear subspace $\Sigma \subset V$ together with functions $x, y \in \bbF_q[\Sigma]$ such that:
\begin{enumerate}
\item The functions $p_2, p_4, q_4, x, y \in \bbF_q[\Sigma]$ generate $\bbF_q[\Sigma]$.
\item The relation $y( xy + 2 q_4 ) = x^3 + p_2 x^2  + p_4 x  + p_6$ holds on $\Sigma$.
\end{enumerate}
\end{theorem}
\begin{proof}
This theorem follows from \cite[Theorem 3.8]{Tho13} when $\bbF_q$ is replaced by a field of characteristic 0. The same proof works over $\bbF_q$, with our restrictions on the characteristic. This is unsurprising, given that the results of Slodowy \cite{Slo80} are proved in positive characteristic with the same restrictions. We explain the construction. Define a matrix
\[ e = \left(
\begin{array}{cccccccc}
 0 & 1 & 0 & 0 & 0 & 1 & 0 & 0 \\
 0 & 0 & 0 & 0 & 1 & 0 & 0 & 2 \\
 0 & 0 & 0 & -1 & 0 & 0 & 0 & 0 \\
 0 & 0 & 0 & 0 & 0 & 0 & 0 & 0 \\
 0 & 0 & -2 & 0 & 0 & 0 & 1 & 0 \\
 0 & 0 & 0 & 0 & 1 & 0 & 0 & -1 \\
 0 & 0 & 0 & -1 & 0 & 0 & 0 & 0 \\
 0 & 0 & -1 & 0 & 0 & 0 & -1 & 0 \\
\end{array}
\right) \]
and a cocharacter $\check{\lambda} \in X_\ast(T')$
\[ \check{\lambda}(t) = \diag(t^2, t, t^{-1}, t^{-2}, 1, t, t^{-1}, 1). \]
Then $\Ad \check{\lambda}(t)(e)= t e$ and $e \in V$ is a subregular nilpotent element. Therefore we can find a unique subregular nilpotent $f \in V$ such that the triple $(e, d \check{\lambda}(2), f)$ is a normal $\frs\frl_2$-triple. We define $\Sigma = e + \frz_\frh(f)^{d \theta = -1}$. 

If $t \in \bbG_m$, then the action $t \cdot v = t \Ad \check{\lambda}(t^{-1})(v)$ leaves $\Sigma$ invariant and contracts $\Sigma$ to the fixed base point $e$; moreover, the morphism $\pi|_\Sigma$ is then $\bbG_m$-equivariant. The functions $x, y \in \bbF_q[\Sigma]$ are chosen to have weight 2 with respect to this action.
\end{proof}
At this point there are two natural discriminant polynomials $\Delta$ in $\bbF_q[V]^G$ that one might consider; the one arising from the usual Lie algebra discriminant in $\frh$, and the discriminant of the polynomial $f(t) = t^4 + p_2 t^3 + p_4 t^2 + p_6 t + q_4^2$, which is used in \S \ref{sec_elliptic_curves_with_marked_points}. In fact, these two functions are equal up to $\bbF_q^\times$-multiple, because they both cut out the same irreducible divisor in $B = \Spec \bbF_q[V]^G$. Since the precise value of $\Delta$ will not be important for us, but rather only its order of vanishing, we will use the symbol $\Delta$ to denote either one of these polynomials in $\bbF_q[V]^G = \bbF_q[p_2, p_4, q_4, p_6]$.

We write $S \to B$ for the natural compactification of $\Sigma$ as a family of projective plane curves given by the equation
\begin{equation}
Y( XY + 2 q_4 Z^2 ) = X^3 + p_2 X^2  Z + p_4 X Z^2 + p_6Z^3.
\end{equation}
We write $O$, $P$, and $Q$ for the three sections of $S - \Sigma$ at infinity given respectively by $[0 : 1 :0 ]$, $[-1: 1 : 0]$ and $[1 : 1 : 0]$. We write $S^\text{rs}$ for the restriction of this family to $B^\text{rs}$. The fundamental relation between the pair $(G, V)$ and this family of curves is as follows:
\begin{theorem}\label{thm_identification_of_stabilizer_and_2-torsion}
\begin{enumerate}
\item The morphism $S \to B$ is smooth exactly above $B^\text{rs}$. Consequently, $S^\text{rs} \to B^\text{rs}$ is a family of smooth, projective, geometrically connected curves. 
\item Let $J_{S^\text{rs}} = \Pic^0_{S^\text{rs} / B^\text{rs}}$ denote the (relative) Jacobian of this family, and let $Z^\text{rs}$ denote the equalizer of the diagram
\[ \xymatrix{ G \times \kappa^\text{rs} \ar@<.7ex>[rr]^(.55){(g, x) \mapsto g \cdot x} \ar@<-.7ex>[rr]_(.55){(g, x) \mapsto x} & & V^\text{rs} }\]
viewed as a finite \'etale group scheme over $\kappa^\text{rs} \cong B^\text{rs}$. Then there is a canonical isomorphism $J_{S^\text{rs}}[2] \cong Z^\text{rs}$ of finite \'etale group schemes over $B^\text{rs}$.
\item Let $k / \bbF_q$ be a field, and let $b \in B^\text{rs}(k)$. Consider the diagram
\[ \xymatrix{ \Sigma_b(k) \ar[r] \ar[d] & G(k) \backslash V_b(k) \ar[d] \\
J_b(k) \ar[r] & H^1(k, J_b[2]), } \]
where the top arrow is the canonical inclusion; the left arrow is the map $R \mapsto [(R) - (O)]$; the right arrow is the injection of Corollary \ref{cor_parameterization_of_orbits}, composed with the isomorphism $H^1(k, Z_G(\kappa_b)) \cong H^1(k, J_b[2]$); and the bottom arrow is the canonical 2-descent map on the Jacobian $J_b$. Then there exists a class $x_b \in H^1(k, J_b[2])$ arising from a trivial orbit such that this diagram commutes up to addition of $x_b$. 
\end{enumerate}
\end{theorem}
\begin{proof}
The first part is established over a field of characteristic 0 in \cite[Corollary 3.16]{Tho13}, using a reduction to \cite{Slo80}, and again the same proof works in our positive characteristic setting. This is not the case for the second part, where the corresponding fact is established in \cite[Corollary 4.12]{Tho13} using analytic techniques. However, the same construction works to show that there is a map $J_{S^\text{rs}}[2] \to Z^\text{rs}$ of local systems of $\bbF_2$-vector spaces on $B^\text{rs}$. To check that it is an isomorphism, it suffices to check that it is an isomorphism on a single stalk, and this can easily be accomplished by lifting to characteristic 0 and applying \cite[Corollary 4.12]{Tho13}.

The third part has been established in characteristic 0 in \cite[Theorem 4.15]{Tho13}, which also shows how to calculate the element $x_b$ using the geometry of the curve $S$. We describe the recipe, although it is not strictly necessary for what we do here. Let $0 \in B(\bbF_q)$ be the central point. Then the curve $S_0$ is a union of three lines. Let $S_0' \subset S_0$ be the branch containing the section $O$ at infinity, and let $E' \in S'_0(\bbF_q) - \{ e \}$ be a rational point. Then there exists a unique $w \in W_0$ such that $w E'$ is conjugate by $G(\overline{\bbF}_q)$ to $\kappa_0$, and for any $b \in B(k)$ we can then take $x_b$ to be the class corresponding to the orbit of $w \kappa_b \in V(k)$.

We still need to extend this result to positive characteristic. However, this is an essentially formal consequence of the first two parts of the theorem, and follows in  exactly the same way as in \cite[\S 4]{Tho13}.
\end{proof}

\begin{theorem}\label{thm_rational_orbits_and_rational_points}
Let $k / \bbF_q$ be a field, and let $b \in B^\text{rs}(k)$. 
\begin{enumerate}
\item The image of the injective map $G(k) \backslash V_b(k) \to H^1(k, J_b[2])$ appearing in Theorem \ref{thm_identification_of_stabilizer_and_2-torsion} contains the canonical image of $J_b(k) / 2 J_b(k)$. 
\item Inside this image, the trivial orbits of $G(k) \backslash V_b(k)$ correspond to the subgroup of $J_b(k) / 2 J_b(k)$ generated by the divisor classes $[(P) - (O)]$ and $[(Q) - (O)]$.
\end{enumerate}
\end{theorem}
\begin{proof}
By Theorem \ref{thm_identification_of_stabilizer_and_2-torsion}, it is enough to prove the second part of the theorem. By definition, the identity of $H^1(k, J_b[2])$ corresponds to the orbit of the Weierstrass section $\kappa_b \in V_b(k)$. We have a short exact sequence of \'etale homology groups (where overline denotes base change to a separable closure $k^s/k$):
\begin{equation}\label{eqn_homology_exact_sequence} \xymatrix@1{ 0 \ar[r] & (\mu_2^3)_{\Sigma = 0} \ar[r] & H_1(\overline\Sigma_b, \bbF_2) \ar[r] & H_1(\overline S_b, \bbF_2) \ar[r] & 0.} 
\end{equation}
Here $(\mu_2^3)_{\Sigma = 0} \subset \mu_2^3$ denotes the kernel of the map which sums up co-ordinates. There is a natural symplectic duality $\langle \cdot, \cdot \rangle$ on $H^1(\overline \Sigma_b, \bbF_2)$ with radical $(\mu_2^3)_{\Sigma = 0}$, and which descends to the Poincar\'e duality pairing on $H_1(\overline S_b, \bbF_2)$. Identifying $J_b[2] = H_1(\overline S_b, \bbF_2)$, this allows us to describe the subgroup of $H^1(k, J_b[2])$ generated by the divisors at infinity as follows: it is the image of $(\mu_2^3)_{\Sigma = 0}^\vee$ under the connecting homomorphism attached to the dual exact sequence of $\bbF_2[\Gamma_k]$-modules (with $\Gamma_k = \Gal(k^s / k)$):
\begin{equation}\label{eqn_cohomology_exact_sequence} \xymatrix@1{ 0 \ar[r] & H_1(\overline S_b, \bbF_2) \ar[r] & H_1(\overline\Sigma_b, \bbF_2)^\vee \ar[r] &  \ar[r]  (\mu_2^3)_{\Sigma = 0}^\vee & 0, } 
\end{equation}
where we use the aforementioned pairing to identify  $H_1(\overline S_b, \bbF_2)^\vee \cong H_1(\overline S_b, \bbF_2)$. We now identify these exact sequences using the representation theory of the pair $(G, V)$. Let $H^\text{sc}$ denote the simply connected cover of $H$, and let $G^\text{sc} = (H^\text{sc})^\theta$; it is a connected subgroup of $H$. Let $C^\text{sc}$ denote the centralizer of $\kappa_b$ in $H^\text{sc}$, and $C$ its image in $H$. Then we can identify $Z_{G^\text{sc}}(\kappa_b) = C^\text{sc}[2]$, $Z_{H^\theta}(\kappa_b) = C[2]$, and $Z_G(\kappa_b) = \im(C^\text{sc}[2] \to C[2])$. The short exact sequence (\ref{eqn_homology_exact_sequence}) is canonically identified with the sequence
\begin{equation} \xymatrix@1{ 0 \ar[r] & Z_{G^\text{sc}} \ar[r] & C^\text{sc}[2] \ar[r] & \im(C^\text{sc}[2] \to C[2]) \ar[r] & 0} 
\end{equation}
(compare \cite[Theorem 4.10]{Tho13} and the proof of Theorem \ref{thm_identification_of_stabilizer_and_2-torsion}). Its dual is canonically identified with the sequence
\begin{equation}\label{eqn_cohomology_exact_sequence_weyl_equivalent} \xymatrix@1{ 0 \ar[r] & Z_{G}(\kappa_b) \ar[r] & C[2] \ar[r] & \pi_0(H^\theta) \ar[r] & 0, } 
\end{equation}
using the Weyl-invariant bilinear form on $X_\ast(C)$ (cf. \cite[Lemma 2.11]{Tho13}) and the canonical identification $C[2] / Z_{G}(\kappa_b) \cong \pi_0(H^\theta)$. The map $W_0 \to \pi_0(H^\theta)$ is an isomorphism, and the composite $W_0 \to \pi_0(H^\theta) \to H^1(k, Z_G(\kappa_b))$ sends an element $w \in W_0$ to the class corresponding to the orbit $G(k) \cdot w \kappa_b$. This concludes the proof.
\end{proof}
The proof of the second part of Theorem \ref{thm_rational_orbits_and_rational_points} has a useful corollary: it gives a criterion to tell when the trivial orbits generate a subgroup of $J_b(k) / 2 J_b(k)$ of order 4 (which one expects to be the case generically). Indeed, taking in mind the identification of the exact sequence (\ref{eqn_cohomology_exact_sequence}) with the sequence (\ref{eqn_cohomology_exact_sequence_weyl_equivalent}), one sees that this should be the case exactly when $H^0(k, Z_G(\kappa_b)) = H^0(k, C[2])$. The action of the Galois group $\Gamma_k$ on $C[2]$ arises from a homomorphism $\Gamma_k \to W(H, C) = W$ giving the action on the torus $C$, and this condition can be described in terms of the image of this homomorphism inside $W$. In particular, in the `generic' case where this image is the whole Weyl group, there will be no additional invariants, and consequently 4 trivial orbits in $G(k) \backslash V_b(k)$.
\begin{corollary}\label{cor_orbits_corresponding_to_selmer_group_exist}
Let $X$ be a smooth, projective, geometrically connected curve over $\bbF_q$, and let $K = \bbF_q(X)$. Let $b \in B^\text{rs}(K)$. Then the subset $G(K) \backslash V_b(K) \subset H^1(K, J_b[2])$ appearing in Corollary \ref{cor_parameterization_of_orbits} (with $k = K$) contains the 2-Selmer group $\Sel_2(J_b)$.
\end{corollary}
\begin{proof}
This follows from the fact that the Hasse principle holds for $G$, i.e.\ that the map $H^1(K, G) \to \prod_v H^1(K_v, G)$ is injective.
\end{proof}

\subsection{$(G, V)$ and local integral orbits}

In the previous section, we have studied rational orbits; we now look at the integral situation. Let $X$ be a smooth, projective, geometrically connected curve over $\bbF_q$, and let $K = \bbF_q(X)$. Let $v$ be a place of $K$, and let $(E, P, Q)$ be tuple consisting of an elliptic curve $E$ over $K_v$ with two distinct, non-trivial marked rational points $P, Q \in E(K_v)$. We assume that the minimal model (as in \S \ref{sec_elliptic_curves_with_marked_points}) of $(E, P, Q)$ has squarefree discriminant, and let $b = (p_2, p_4, q_4, p_6) \in B(\cO_{K_v})$ denote the associated set of invariants. We write $J_b$ for the Jacobian of $E$, which we identify with $E$ via the map $R \mapsto (R) - (O)$.
\begin{theorem}\label{thm_properties_of_integral_square_free_orbits}
With assumptions as above, let $\cJ_b$ denote the N\'eron model of $E$ over $\cO_{K_v}$. Then:
\begin{enumerate}
\item The map $H^1(\cO_{K_v}, \cJ_b[2]) \to H^1(K_v, J_b[2])$ in \'etale cohomology is injective.
\item An orbit in $G(K_v) \backslash V_b(K_v)$ admits an integral representative (i.e.\ intersects $V_b(\cO_{K_v})$) if and only if it corresponds to an element of $J_b(K_v) / 2 J_b(K_v)$.
\item Suppose that $x, y \in V_b(\cO_{K_v})$ and $\gamma \in G(K_v)$ satisfies $\gamma x = y$. Then $\gamma \in G(\cO_{K_v})$. 
\end{enumerate}
\end{theorem}
\begin{proof}
We have $H^1(\cO_{K_v}, \cJ_b[2]) = H^1(k(v), \cJ_b[2](\kappa(v)))$. Since the discriminant is square-free, we have $\cJ_b[2](\kappa(v)) = J_b[2](K_v^s)^{I_{K_v}}$, so the injectivity of the first part is a consequence of inflation-restriction.

For the `if' of the second part, we use the existence of the section $\Sigma \subset V$, which shows (together with the commutative diagram of Theorem \ref{thm_identification_of_stabilizer_and_2-torsion}) that any element of $J_b(K_v) / 2 J_b(K_v)$ which can be represented by a divisor $(R) - (O)$, where $R \in \Sigma_b(\cO_{K_v})$, is represented by an element of $V(\cO_{K_v})$. Since the trivial orbits have integral representatives, essentially by definition, this reduces us to showing that any non-trivial orbit in $J_b(K_v) / 2 J_b(K_v)$ is represented by such a divisor $(R) - (O)$. We have a short exact sequence
\[ \xymatrix@1{ 0 \ar[r] & \cJ_b(\cO_{K_v})^0 \ar[r] & J_b(\cO_{K_v}) \ar[r] & \cJ_{b}(k(v)) \ar[r] & 0,} \]
where the kernel is a pro-$p$-group ($p = \cha \bbF_q$), hence an isomorphism 
\[ \cJ_b(\cO_{K_v}) /  2\cJ_b(\cO_{K_v}) \cong \cJ_b(k(v)) / 2 \cJ_b(k(v)) \cong H^1(\cO_{K_v}, \cJ_b[2]). \]
If $[ \overline{x} ] \in  \cJ_b(k(v)) / 2 \cJ_b(k(v))$ is a non-trivial class (i.e.\ not in the subgroup generated by the 3 marked points of $E$ at infinity), we can choose a representative $\overline{x} \in \cJ_b(k(v))$ of the form $(\overline{R}) - (\overline{O})$, where $\overline{R} \in \Sigma_b(k(v))$. Lifting $\overline{R}$ to a point $R \in \Sigma_b(\cO_{K_v})$ via Hensel's lemma then shows the existence of the desired integral representative in $V_b(\cO_{K_v})$.

We now turn to the `only if' of the second part. We first note that any element $x \in V_b(\cO_{K_v})$ in fact lies in $V_b^\text{reg}(\cO_{K_v})$, i.e.\ $\overline{x} = x \text{ mod }(\varpi_v)$ is regular in $V_{k(v)}$. This is clear if $\Delta(v)$ is a unit in $\cO_{K_v}$. Otherwise, we note that $\overline{x}$ is regular in $V_{k(v)}$ if and only if it is regular in $\frh_{k(v)}$; and if it is not regular in $\frh_{k(v)}$, then its centralizer has dimension at least $\dim T + 2$ (see \cite[III. 3.25]{Spr70}). Let $\frc = \frz_{\frh_{K_v}}(x)$, $\frc^0 = \frc \cap \frh_{\cO_{K_v}}$.
 Let $f : \frh_{\cO_{K_v}} / \frc^0 \to \frh_{\cO_{K_v}} /  \frc^0$ denote the morphism induced by $\ad x$ after passage to quotient. We have the relation $\det f = \Delta(x)$, up to units in $\cO_{K_v}^\times$. If $\overline{x}$ is not regular, then $\overline{f} = f \text{ mod } (\varpi_v)$ has kernel of dimension at least 2, hence $\ord_{K_v} \det f \geq 2$, a contradiction.

We next observe that the map $G_{\cO_{K_v}} \to V_b^\text{reg}$, $g \mapsto g \cdot \kappa_b$, is \'etale, and a torsor over its image $V_b^\text{reg, 0} \subset V_b^\text{reg}$ for the \'etale group scheme $Z_{G_{\cO_{K_v}}}(\kappa_b)$ over $\cO_{K_v}$. Moreover, we have $V_b^\text{reg} = \cup_{w \in W_0} w \cdot V_b^\text{reg, 0}$ (by \cite[Theorem 0.17]{Lev07}). It follows that there is a canonical bijection
\begin{equation}\label{eqn_regular_integral_orbits} G(\cO_{K_v}) \backslash V_b^\text{reg, 0}(\cO_{K_v}) \cong H^1(\cO_{K_v},Z_{G_{\cO_{K_v}}}(\kappa_b)). 
\end{equation}
The isomorphism $Z_G(\kappa_b) \cong J_b[2]$ extends uniquely to an isomorphism $Z_{G_{\cO_{K_v}}}(\kappa_b) \cong \cJ_b[2]$. If $\Delta$ is a unit, then this is immediate from Theorem \ref{thm_identification_of_stabilizer_and_2-torsion}. If $\Delta$ is not a unit, then it suffices to show that the isomorphism $Z_G(\kappa_b) \cong J_b[2]$ identifies $Z_G(\kappa_b)(\kappa(v)) \subset Z_G(\kappa_b)(K_v^s)^{I_v}$ with $\cJ_b[2](\kappa(v)) \subset J_b[2](K_v^s)^{I_v}$. Since the latter group has order 2, it is enough to show that $Z_G(\kappa_b)(\kappa(v))$ is non-trivial. This follows from the fact that $\Sigma_{\overline{b}}$ has a unique singularity of type $A_1$, as we now show. Let $\overline{b} = b \text{ mod }(\varpi_v)$. The element $\kappa_{\overline{b}} \in V(k(v))$ has a Jordan decomposition $\kappa_{\overline{b}} = v_s + v_n$ as a sum of commuting semi-simple and nilpotent parts, and we can compute (using the same technique as in \cite[Proposition 2.8]{Tho13})
\begin{equation}\label{eqn_calculation_of_regular_non_semisimple_centralizer} Z_G(\kappa_{\overline{b}}) =  Z_{Z_{G^\text{sc}}(v_s)} [2] / Z_{G^\text{sc}}.
\end{equation}
 The fact that $\Sigma_{\overline{b}}$ has a singularity of type $A_1$ implies (\cite[Corollary 3.16]{Tho13}, the proof of which goes over without change in our setting) that $Z_{G^\text{sc}}(v_s)$ has derived group of type $A_1$. In particular, its centre contains a torus of rank 3, and consequently the group appearing in (\ref{eqn_calculation_of_regular_non_semisimple_centralizer}) must be non-trivial.

We can thus enlarge (\ref{eqn_regular_integral_orbits}) to a commutative diagram
\begin{equation}
\begin{aligned}
 \xymatrix{ G(\cO_{K_v}) \backslash V_b^\text{reg, 0}(\cO_{K_v})  \ar[r] \ar[d] &  H^1(\cO_{K_v}, \cJ_b[2]) \ar[d] \\
 G(K_v) \backslash V_b(K_v) \ar[r] & H^1(K_v, J_b[2]). } 
\end{aligned}
\end{equation}
This shows that any element of  $G(K_v) \backslash V_b(K_v)$ which is in the image of the left-hand vertical arrow lies in the image of $H^1(\cO_{K_v}, \cJ_b[2]) \cong J_b(K_v) / 2 J_b(K_v) \subset  H^1(K_v, J_b[2])$. Since we have $V_b^\text{reg} = \cup_{w \in W_0} w V_b^\text{reg, 0}$, and $w$ acts on $H^1(K_v, J_b[2])$ as translation by trivial orbits, we finally see that any element of $G(K_v) \backslash V_b(K_v)$ which admits an integral representative corresponds to an element of $J_b(K_v) / 2 J_b(K_v)$.

Finally, we come to the third part of the theorem. The integrality is insensitive to passage to unramified extensions of $K_v$, so we reduce to the statement that the \'etale group scheme $Z_{H^\theta_{\cO_{K_v}}}(\kappa_b)$ satisfies the N\'eron mapping property, i.e.\ its $K_v$-points all extend to $\cO_{K_v}$-points. If $\Delta$ is a unit then this \'etale group scheme is finite \'etale. If $\ord_{K_v} \Delta = 1$, then we have seen that the action of inertia on $Z_{G_{\cO_{K_v}}}(\kappa_b)(K_v^s)$ is non-trivial, so $|Z_{H^\theta_{\cO_{K_v}}}(\kappa_b)(K_v)| \leq 2^3$. On the other hand, using again the Jordan decomposition $\kappa_{\overline{b}} = v_s + v_n$, we have
\[ Z_{H^\theta}(\kappa_{\overline{b}})(\kappa(v)) = Z_{Z_{H^\theta}(v_s)}(\kappa(v))[2], \]
and this group has size at least $2^3$. This shows that the desired property of $Z_{H^\theta_{\cO_{K_v}}}(\kappa_b)$ does hold, and completes the proof of the theorem.
\end{proof}
\subsection{$(G, V)$ and global integral orbits}\label{sec_global_integral_orbits}
We can now discuss the global picture. Let $X$ be a smooth, projective, geometrically connected curve over $\bbF_q$, and let $K = \bbF_q(X)$. Let $D = \sum_v m_v \cdot v$ be a divisor on $X$, and let $(E, P, Q) \in \cX_D$. We recall (see \S \ref{sec_elliptic_curves_with_marked_points}) that this means that $E$ is an elliptic curve over $K$ with two distinct non-trivial marked rational points $P, Q \in E(K)$, and which can be represented by an equation
\begin{equation}\label{eqn_global_weierstrass_equation} y( xy + 2 q_4 ) = x^3 + p_2 x^2  + p_4 x  + p_6 
\end{equation}
with 
\begin{equation}b = (p_2, p_4, q_4, p_6) \in H^0( X, \cO_X(2D) \oplus \cO_X(4D) \oplus \cO_X(4D) \oplus \cO_X(6D)) = H^0(X, B_D) \subset B(K)
\end{equation}
of square-free discriminant in $H^0(X, \cO_X(24D))$. (The reason for restricting to curves with $\cL_E$ a square is that the invariant degrees of the representation $(G, V)$ then agree with the weights of the equation (\ref{eqn_global_weierstrass_equation}) defining the curve $E$.)

Let $x \in V_b(K)$ be an element corresponding to an element of the group $\Sel_2(E)$ (see Corollary \ref{cor_orbits_corresponding_to_selmer_group_exist}). Then for every place $v$ of $K$, $\varpi_v^{m_v} x$ has minimal, integral invariants $\pi( \varpi_v^{m_v} x ) = \varpi_v^{m_v} \cdot b \in \cO_{K_v}^4$ of squarefree discriminant, and Theorem \ref{thm_properties_of_integral_square_free_orbits} implies that we can find $g_v \in G(K_v)$ such that $ \varpi_v^{m_v} x \in g_v V(\cO_{K_v})$. For almost all places $v$, we have $m_v = 0$ and can choose $g_v = 1$. Moreover, $g_v$ is defined up to right multiplication by $G(\cO_{K_v})$, by the third part of Theorem \ref{thm_properties_of_integral_square_free_orbits}. If we replace $x$ by $\gamma x$ for some $\gamma \in G(K)$, then $g_v$ can be replaced by $\gamma g_v$. We have therefore defined a map 
\begin{equation} \inv : \Sel_2(E) \to G(K) \backslash G(\bbA_K) / G(\widehat{\cO}_K).
\end{equation}
 It is clear that this map depends only on $(E, P, Q)$ and not on the choice of equation $b \in H^0(X, \cO_X(D))$ representing $(E, P, Q)$ (since all choices differ by the action of $\bbF_q^\times$).

To any $g \in G(\bbA_K)$, we associate the $G$-torsor $F_g$ and the vector bundle $\cV_g = F_g \times_G V$, which has sections described by (\ref{eqn_sections_of_bundle_associated_to_torsor}).  The above discussion shows that if $[g] = \inv(x)$, then $x$ naturally defines a element of 
\[ V(K) \cap \prod_v g_v \varpi_v^{-m_v} V(\cO_{K_v}) = H^0(X,  \cV_g(D)), \]
and the image of $x$ under the map $\pi : H^0(X, \cV_g(D)) \to H^0(X, B_D)$ equals $b$. We have constructed the first map in the following:
\begin{theorem}\label{thm_invariant_partition_of_selmer_elements}
Let $(E, P, Q) \in \cX_D$ be represented by $b \in H^0(X, B_D)$. We identify $E = J_b$ using the map $R \mapsto (R) - (O)$. Let $g = (g_v)_v \in G(\bbA_K)$. Then the following two sets are in canonical bijection:
\begin{enumerate}
\item The set of elements $x \in \Sel_2(E)$ such that $\inv(x) = [(g_v)_v]$.
\item The set of sections $s \in H^0(X, \cV_g(D))$ such that $\pi(s) = b$, taken up to the action of the group $\Aut(F_g)$. 
\end{enumerate}
\end{theorem}
\begin{proof}
We have constructed the map from the first set to the second set. We now construct its inverse. Let $s$ be a global section in
\[ H^0(X, \cV_g(D)) = V(K) \cap \prod_v g_v \varpi_v^{-m_v} V(\cO_{K_v}) \]
such that $\pi(s) = b$. Writing $x$ for the image of $s$ in $V(K)$ under the canonical inclusion, we obtain an orbit in $G(K) \backslash V_b(K)$. This orbit is independent of the choice of representative in the $\Aut(F_g)$-orbit of $s$; indeed, we have $\Aut(F_g) = G(K) \cap g G(\widehat{\cO}_K) g^{-1}$, so replacing $s$ by $\gamma s$ for $\gamma \in \Aut(F_g)$ would just replace $x$ by $\gamma x$, leaving the $G(K)$-orbit of $x$ unchanged.

We need to show that $x$ lies in the subset of $G(K) \backslash V_b(K)$ corresponding to the 2-Selmer group. However, this follows from the second part of Theorem \ref{thm_properties_of_integral_square_free_orbits} and the fact that $s$ has square-free discriminant. It is clear from the construction that this map is inverse to the other, so this completes the proof.
\end{proof}
To illustrate the construction of this invariant map, we calculate its image when applied to the trivial elements in $J_b(K) / 2 J_b(K) \subset \Sel_2(J_b)$. Recall that we have defined $\kappa = E + \frz_\frh(F)$, where $(E, d \check{\rho}(2), F)$ is a regular normal $\frs\frl_2$-triple in $\frh$. The action $t \cdot x = t \Ad \check{\rho}(t^{-1})(x)$ leaves $\kappa$ invariant and contracts to the unique fixed point $E$ (see Proposition \ref{prop_existence_of_kostant_section}). In particular, if $v$ is a place of $K$, $b \in B(K_v)$, and $\lambda \in K_v^\times$, then we have the following formula giving the behaviour of the Kostant section under scaling:
\begin{equation} \kappa_{\lambda b} = \check{\rho}(\lambda^{-1}) \lambda \kappa_b. 
\end{equation}
If $b \in B(\cO_{K_v})$, then $\kappa_b \in V(\cO_{K_v})$ is an integral representative of the orbit in $V_b(K)$ corresponding to the identity element of $\Sel_2(J_b)$. If $b \in H^0(X, B_{D}) \subset B(K)$ is associated to a pointed curve as above, then we find $\varpi_v^{m_v} b \in B(\cO_{K_v})$ is the minimal integral representative, hence
\begin{equation} \kappa_{\varpi_v^{m_v} b} = \check{\rho}(\varpi_v^{-m_v}) \varpi_v^{m_v} b \in \varpi_v^{m_v} V(\cO_{K_v}). 
\end{equation}
It then follows from the definition that we have $\inv(\kappa_b) = [(\check{\rho}(\varpi_v^{m_v}))_v]$. The same formalism applies to the other trivial orbits: if $w \in W_0$, then the representative of the corresponding trivial orbit in $V(K)$ is $w \kappa_b$. For each place $v$ of $K$, we have
\begin{equation} w \check{\rho}(\varpi_v^{-m_v}) w^{-1} \varpi_v^{m_v} w \kappa_b= w \check{\rho}(\varpi_v^{-m_v}) \varpi_v^{m_v} \kappa_b \in \varpi_v^{m_v} V(\cO_{K_v}), \end{equation}
so it follows from the definition that we have $\inv(w \kappa_b ) = [ (w \check{\rho}(\varpi_v^{m_v}) w^{-1})_v]$. This implies in particular:
\begin{lemma}\label{lem_geography_of_trivial_orbits}
Let $(E, P, Q) \in \cX_D$ be represented by $b \in H^0(X, B_D)$, and let $x \in \Sel_2(J_b)$ be a trivial element. Suppose that $\deg D > 0$. Then $\inv(x) \in \cY_{G, B}(D)^{<0}$.
\end{lemma}
\begin{proof}
We just treat the case of $\kappa_b$, since the other cases are very similar. We need to show that $g = ((\check{\rho}(\varpi_v^{m_v}))_v)$ lies in $B(\bbA_K)^{\text{pos, ss}}$ and that the `lowest slope' part of $\cV_g(D)$ has strictly negative slope. Since the Levi quotient of $B$ is a torus, the semi-stability condition is vacuous, so we must check is that for all $a \in R^-$, we have $\log_q \| a((\check{\rho}(\varpi_v^{m_v}))_v) \| > 0$. We compute
\[ \log_q \| a((\check{\rho}(\varpi_v^{m_v}))_v) \| = - \langle \check{\rho}, a \rangle \cdot \deg D > 0. \]
The lowest slope part of $\cV_g(D)$ has slope 
\[ \log_q \| \alpha_0((\check{\rho}(\varpi_v^{m_v}))_v) \| = - \langle \check{\rho}, \alpha_0 \rangle \cdot \deg D + \deg D = - 2 \deg D < 0,  \]
as required.
\end{proof}

\subsection{The main theorem}

We once again suppose that$X$ be a smooth, projective, geometrically connected curve over $\bbF_q$, and let $K = \bbF_q(X)$. If $D$ is a divisor on $X$, then we write $H^0(X, B_D)^\text{sf} \subset H^0(X, B_D)$ for the set of elements of square-free discriminant $\Delta \in H^0(X, \cO_X(24D))$. Then (Corollary \ref{cor_moduli_of_pointed_curves}) there is a surjection $H^0(X, B_D)^\text{sf} \to \cX_D$, the fibre above a given isomorphism class  $[(E, P, Q)]$  having cardinality equal to $\bbF_q^\times \cdot |\Aut(E, P, Q)|^{-1}$. If $g = [(g_v)_v] \in \cY_G$, then we write $H^0(X, \cV_g(D))^\text{sf} \subset H^0(X, \cV_g(D))$ for the pre-image of $H^0(X, B_D)^\text{sf}$. We also write $H^0(X, \cV_g(D))^{\text{sf,nt}} \subset H^0(\cV_g(D))^\text{sf}$ for the set of elements of $ H^0(\cV_g(D))^\text{sf}$ which are non-trivial when viewed inside $V(K)$ (in the sense of  Lemma \ref{lem_trivial_orbits}).
\begin{proposition} \label{prop_existence_and_equality_of_local_squarefree_densities}
Let $g = (g_v)_v \in G(\bbA_K)$.
\begin{enumerate}
\item The limit 
\[ \delta_B = \lim_{\deg D \to \infty} \frac{|H^0(X, B_D)^\text{sf}|}{|H^0(X, B_D)|} \]
exists and is strictly positive.
\item The limit
\[ \delta_V = \lim_{\deg D \to \infty} \frac{|H^0(X, \cV_g(D))^\text{sf}|}{|H^0(X, \cV_g(D))|} \]
exists and is strictly positive, and does not depend on $g$.
\item We have $ \int_{g \in G(\widehat{\cO}_K)} \,d \tau_G \delta_B = q^{12(g_X-1)} \delta_V $, where $\tau_G$ denotes the Tamagawa measure on $G(\bbA_K)$.
\end{enumerate} 
\end{proposition}
\begin{proof}
If $v$ is a place of $K$, define
\[ \alpha_v = \frac{ |\{ x \in B(\cO_{K_v} / (\varpi_v^2)) \mid \Delta(x) \equiv 0 \text{ mod }\varpi_v^2 \}| }{ q_v^8 } \]
and
\[  \beta_v = \frac{ |\{ x \in V(\cO_{K_v} / (\varpi_v^2)) \mid \Delta(x) \equiv 0 \text{ mod }\varpi_v^2 \}| }{ q_v^{32} }. \]
In \cite[\S 5.1]{Ngo14} it is proved using results of Poonen \cite{Poo03} that the limit $\delta_V$ exists and equals $\prod_v (1 - \beta_v)$. A similar argument using the results of \cite{Poo03} shows that the limit $\delta_B$ exists and equals $\prod_v (1 - \alpha_v)$. It is easy to see that both of these products are strictly positive. To finish the proof of the proposition, we need to show that $ \int_{g \in G(\widehat{\cO}_K)} \,d \tau_G \delta_B = q^{12(g_X-1)} \delta_V $, or even (using the definition of the Tamagawa measure) that $ \int_{g \in G(\cO_{K_v})} | \omega_G |_v  (1 - \alpha_v) =(1 - \beta_v)$ for each place $v$ of $K$, $\omega_G$ being an invariant differential form of top degree on $G$ (over $\bbF_q$). We will establish this using an integral formula.

Let $\omega_V$ and $\omega_G$ be invariant differential forms of top degree on $V$ and $G$, respectively. Let $\omega_B = dp_2 \wedge dp_4 \wedge dq_4 \wedge dp_6$, a differential form of top degree on $B$. Let $\varphi : B(K_v) \to \bbR$ denote the characteristic function of the open subset of $b \in B(\cO_{K_v})$ where $\ord_{K_v} \Delta(b) \leq 1$. Let $f : V(K_v) \to \bbR$ denote the characteristic function of the open subset of $x \in V(\cO_{K_v})$ where $\ord_{K_v} \Delta(x) \leq 1$. Then we must show the identity
\[\int_{g \in G(\cO_{K_v})} | \omega_G |_v \int_{b \in B(K_v)} \varphi(b) \, | \omega_B |_v =  \int_{x \in V(K_v)} f(x) \, | \omega_V |_v  . \]
If $\frc \subset V_{K_v}$ is a Cartan subspace, we write $\mu_\frc : G_{K_v} \times \frc \to V_{K_v}$ for the action map. Exactly the same argument as in \cite[Proposition 2.13]{Tho14} shows that for any Cartan subspace $\frc \subset V_{K_v}$, we have an identity 
\[ \mu_\frc^\ast \omega_V = \lambda \omega_G \wedge \pi|_\frc^\ast \omega_B \]
for some scalar $\lambda \in \bbF_q^\times$ which is independent of the choice of Cartan subspace.

Let $\frc_1, \dots, \frc_s \subset V_{K_v}$ denote representatives for the distinct $G(K_v)$-conjugacy classes of Cartan subspaces. Each element $v \in V^\text{rs}(K_v)$ is contained in a unique Cartan subspace, so we obtain an identity
\[ \int_{x \in V(\cO_{K_v})} f(x) \, | \omega_V |_v = \sum_{i=1}^s \int_{(g, c_i) \in G(K_v) \times \frc_i} \frac{f(g c_i)}{N_G(\frc_i)(K_v)} \, | \lambda |_v  | \omega_G \wedge \pi|_\frc^\ast \omega_B |_v. \]
Let $\frc_i^0 = \frc_i \cap [ G(K_v) \cdot V(\cO_{K_v}) ]$, an open subset of $\frc_i$. It follows from Theorem \ref{thm_properties_of_integral_square_free_orbits} and the invariance of the measure $ |\omega_G |_v$ that this last integral is equal to
\[ \sum_{i=1}^s \int_{g \in G(\cO_{K_v})} | \omega_G |_v \int_{c_i \in \frc_i} \frac{\varphi(\pi(c_i))}{N_G(\frc_i)(K_v)} | \lambda |_v |\pi|_\frc^\ast \omega_B |_v \]
\[ =  \sum_{i=1}^s \int_{g \in G(\cO_{K_v})} | \omega_G |_v  | N_G(\frc_i)(K_v) |^{-1} \int_{b \in B(K_v)} \varphi(b) |\frc_{i, b}(K_v) \cap \frc_i^0| |\omega_B|_v. \]
To finish the proof, we therefore just need to show that if $b \in B(\cO_{K_v})$ satisfies $\ord_K \Delta(b) \leq 1$, then
\[ \sum_{i=1}^s|\frc_{i, b}(K_v) \cap \frc_i^0| \times |  N_G(\frc_i)(K_v) |^{-1} = 1. \]
The left-hand side counts the number of $G(K_v)$-orbits in $V_b(K_v)$ which have an integral representative, each orbit being weighted by $|Z_G(\kappa_b)(K_v)|^{-1}$. The total number of orbits equals $| J_b(K_v) / 2 J_b(K_v)  |= |J_b(K_v)[2]|$, by Theorem \ref{thm_properties_of_integral_square_free_orbits}. This quantity in turn is equal to $|Z_G(\kappa_b)(K_v)|$, by Theorem \ref{thm_identification_of_stabilizer_and_2-torsion}. This completes the proof.
\end{proof}
We now come to the first main theorem of this paper. If $D$ is a divisor on $X$ and $(E, P, Q) \in \cX_D$, we write $A_{(E, P, Q)} \subset \Sel_2(E)$ for the trivial subgroup generated by the classes of $P$ and $Q$, and $\Sel_2(E)^\text{nt} = \Sel_2(E) - A_{(E, P, Q)}$ for its complement.
\begin{theorem}\label{thm_average_of_non_trivial_selmer}
The limit \[ \lim_{\deg D \to \infty} \sum_{(E, P, Q) \in \cX_D} \frac{  |\Sel_2(E)^{\text{nt}}| \cdot | \Aut(E, P, Q) |^{-1} \cdot | E(K)[2] |^{-1}} { |\cX_D| } \]
exists and equals 8.
\end{theorem}
\begin{proof}
By Corollary \ref{cor_moduli_of_pointed_curves}, we have
\[ \lim_{\deg D \to \infty} \frac{| \cX_D |}{H^0(X, B_D)^\text{sf}} = (q-1)^{-1}. \]
By Theorem \ref{thm_invariant_partition_of_selmer_elements}, we have
\[ (q-1) \sum_{(E, P, Q) \in \cX_D} \frac{ |\Sel_2(E)^{\text{nt}}| \cdot | \Aut(E, P, Q) |^{-1} \cdot | E(K)[2] |^{-1}} { | H^0(X, B_D)^\text{sf} | } = \int_{g \in \cY_G} \frac{|H^0(X, \cV_g(D))^{\text{sf, \text{nt}}}|}{|H^0(X, B_D)^{\text{sf}}|} \, d \nu_G \]
\[ = \int_{\cY_G} \frac{ |H^0(X, \cV_g(D))^{\text{sf, nt}}|}{|H^0(X, B_D)|} \times \frac{|H^0(X, B_D)|}{|H^0(X, B_D)^{\text{sf}}|} \, d \nu_G,\]
hence
\[  \lim_{\deg D \to \infty}\sum_{(E, P, Q) \in \cX_D} \frac{|\Sel_2(E)^{\text{nt}}| \cdot | \Aut(E, P, Q) |^{-1} \cdot | E(K)[2] |^{-1}} {|\cX_D| } = \delta_B^{-1}\times  \lim_{\deg D \to \infty} \int_{\cY_G} \frac{|H^0(X, \cV_g(D))^{\text{sf, nt}}|}{|H^0(X, B_D)|} \, d \nu_G. \]
We would like to compute the pointwise limit of the integrand and then interchange the order of the integral and the limit. This can be justified only after a process of `cutting off the cusp'. Applying the decomposition of \S \ref{sec_semistability_and_integration}, we get
\[ \int_{\cY_G} \frac{|H^0(X, \cV_g(D))^\text{sf, nt}|}{|H^0(\cX, B_D)|} \, d \nu_G = \sum_P \int_{ \cY_{G, P}}  \frac{|H^0(X, \cV_g(D))^\text{sf, nt}|}{|H^0(X, B_D)|} \, d \nu_G \]
\[ = \sum_P \left[ \int_{ \cY_{G, P}(D)^{> 2g_X - 2}}  \frac{|H^0(X, \cV_g(D))^\text{sf, nt}|}{|H^0(X, B_D)|} \, d \nu_G  + \int_{ \cY_{G, P}^\text{sp}}  \frac{|H^0(X, \cV_g(D))^\text{sf, nt}|}{|H^0(X, B_D)|} \, d \nu_G + \int_{ \cY_{G, P}^{< 0}}  \frac{|H^0(X, \cV_g(D))^\text{sf, nt}|}{|H^0(X, B_D)|} \, d \nu_G \right] , \]
where the sums are over the set of standard parabolic subgroups of $G$. (We recall that these are the parabolics containing the Borel subgroup $B \subset G$ corresponding to the set $R^- = \{ - a_1, -a_2, -a_3, a_4 \}$ of simple roots of $G$.) Applying Lemma \ref{lem_geography_of_trivial_orbits}, we see that when $\deg D > 0$, this equals
\[ \sum_P \left[ \int_{ \cY_{G, P}(D)^{> 2g_X - 2}}  \frac{|H^0(X, \cV_g(D))^\text{sf}|}{|H^0(X, B_D)|} \, d \nu_G  + \int_{ \cY_{G, P}^\text{sp}}  \frac{|H^0(X, \cV_g(D))^\text{sf}|}{|H^0(X, B_D)|} \, d \nu_G + \int_{ \cY_{G, P}^{< 0}}  \frac{|H^0(X, \cV_g(D))^\text{sf, nt}|}{|H^0(X, B_D)|} \, d \nu_G \right]. \]
We will see that the terms corresponding to $\cY_{G, P}(D)^{> 2g_X - 2}$ dominate, while the others vanish in the limit. Note that for any $g\in \cY_{G, P}$, we have $g \in \cY_{G, P}(D)^{> 2g_X - 2}$ for all divisors $D$ of sufficiently large degree (depending on $g$). For divisors of degree greater than $2g_X - 2$ we have $|H^0(X, B_D)| = q^{4(1-g_X) + 16 \deg D}$, and an application of Corollary \ref{cor_vanishing_sections_of_lowest_slope_part} shows that such $D$ we have
\[ \int_{ \cY_{G, P}(D)^{> 2g_X - 2}}  \frac{|H^0(X, \cV_g(D))^\text{sf}|}{|H^0(X, B_D)|} \, d \nu_G = \int_{ \cY_{G, P}(D)^{> 2g_X - 2}}\frac{|H^0(X, \cV_g(D))|}{|H^0(X, B_D)|}\frac{|H^0(X, \cV_g(D))^\text{sf}|}{|H^0(X, \cV_g(D))|} \, d \nu_G. \]
\[ = q^{12(1-g_X)} \int_{ \cY_{G, P}(D)^{> 2g_X - 2}}
\frac{|H^0(X, \cV_g(D))^\text{sf}|}{|H^0(X, \cV_g(D))|} \, d \nu_G.\]
The integrand in this expression is bounded by 1, and as $\deg D \to \infty$ its value tends to a limit $\delta_V$ which is independent of the choice of $g$, by Proposition \ref{prop_existence_and_equality_of_local_squarefree_densities}. Applying the dominated convergence theorem, we find that 
\[ \lim_{\deg D \to \infty}\int_{ \cY_P(D)^{> 2g_X - 2}}  \frac{|H^0(X, \cV_g(D))^\text{sf}|}{|H^0(X, B_D)|} \, d \nu_G = q^{12(1-g_X)} \delta_V \int_{ \cY_{G, P}} \, d \nu_G. \]
To take care of the contribution in special range, we calculate using Corollary \ref{cor_reducible_orbits_in_canonical_reduction} and Lemma \ref{lem_integration_of_boundary_terms}:
\[ \int_{ \cY_{G, P}^\text{sp}}  \frac{|H^0(X, \cV_g(D))^\text{sf}|}{|H^0(X, B_D)|} \, d \nu_G \leq  \int_{ \cY_{G, P}^\text{sp}}  \frac{|H^0(X, \cV_g(D))|}{|H^0(X, B_D)|} \, d \nu_G = O\left( \sum_{\substack{\sigma \in \Lambda_P^\text{pos} \\ \deg D + \langle \sigma, \alpha_0 \rangle \in [0, 2g_X - 2]}} q^{-\langle \sigma, \delta_P \rangle}\right), \]
where the implied constant depends only on $X$. This tends to 0 as $\deg D \to \infty$. To take care of the remaining contributions, we note that Corollary \ref{cor_parabolic_reducibility_from_canonical_reduction} implies that
\[ \int_{ \cY_{G, P}^{< 0}}  \frac{|H^0(X, \cV_g(D))^\text{sf, nt}|}{|H^0(X, B_D)|} \, d \nu_G = 0 \]
unless $P = B$ or the Levi quotient of $P$ has semisimple rank 1. In these cases we will show that
\begin{equation}\label{eqn_limit_of_boundary_term_equals_zero} \lim_{\deg D \to \infty}\int_{ \cY_{G, P}^{< 0}}  \frac{|H^0(X, \cV_g(D))^\text{sf, nt}|}{|H^0(X, B_D)|} \, d \nu_G= 0. 
\end{equation}
Let us first treat the (harder) case of $P = B$. Let $\cC$ denote the set of non-empty subsets $M \subset \Phi_V$ which are closed under the relation $\geq$: i.e.\ if $a \in \Phi_V$, $b \in M$, and $a \geq b$, then $a \in M$. Note that $\alpha_0 \in M$ for all $M \in \cC$. Then we have $\cY_{G, B}(D)^{<0} = \sqcup_{M \in \cC} \cY_{G, B}(D)^{<0, M}$, where we define $\cY_{G, B}(D)^{<0, M}$ to be the set of $G$-torsors $F \in \cY_{G, B}(D)^{<0}$ such that for $a \in \Phi_V$, the slope $\sigma_{F_B}$ of the canonical reduction $F_B$ satisfies $\langle \sigma_{F_B}, a \rangle + D < 0$ if and only if $a \in M$. This allows us to decompose
\begin{equation}\label{eqn_hasse_decomposition_of_borel_boundary_term} \int_{ \cY_{G, P}^{< 0}}  \frac{|H^0(X, \cV_g(D))^\text{sf, nt}|}{|H^0(X, B_D)|} \, d \nu_G = \sum_{M \in \cC} \int_{ \cY_{G, P}^{< 0, M}}  \frac{|H^0(X, \cV_g(D))^\text{sf, nt}|}{|H^0(X, B_D)|} \, d \nu_G. 
\end{equation}
Let $\cC_0 \subset \cC$ denote the set of subsets $M \in \cC_0$ not containing any of the sets $S$ appearing in the statement of Corollary \ref{cor_identification_of_reducible_or_weierstrass_orbits}; then the summand in (\ref{eqn_hasse_decomposition_of_borel_boundary_term}) corresponding to $M \in \cC$ can be non-zero only if $M \in \cC_0$. To show (\ref{eqn_limit_of_boundary_term_equals_zero}) in  case $P = B$, it is therefore enough to show that the equality
\begin{equation}\label{eqn_vanishing_of_borel_boundary_term}
\lim_{\deg D \to \infty}\int_{ \cY_{G, B}^{< 0, M}}  \frac{|H^0(X, \cV_g(D))|}{|H^0(X, B_D)|} \, d \nu_G = 0. 
\end{equation}
holds for each $M \in \cC_0$. If $M \in \cC$ and $ \cY_{G, B}^{<0, M}$, then Corollary \ref{cor_reducible_orbits_in_canonical_reduction} implies that we have
\[  \frac{|H^0(X, \cV_g(D))|}{|H^0(X, B_D)|} = O(q^{-|M| \deg D - \langle \sigma,  \sum_{a \in M} a \rangle}), \]
where the implied constant depends only on $X$. Combining this with Lemma \ref{lem_integration_of_boundary_terms}, we get for any $M \in \cC$:
\begin{equation}\label{eqn_borel_boundary_term} \int_{ \cY_{G, B}^{< 0, M}}  \frac{|H^0(X, \cV_g(D))|}{|H^0(X, B_D)|} \, d \nu_G = O\left( \sum_{\substack{\sigma \in \Lambda_B^\text{pos} \\ \forall a \in M, \langle \sigma, a \rangle + \deg D < 0 \\
\forall a \in \Phi_V - M, \langle \sigma, a \rangle + \deg D \geq 0}} q^{ - |M| \deg D - \langle \sigma, \delta_B + \sum_{a \in M} a \rangle}\right), 
\end{equation}
where the implied constant again depends only on $X$. If $a \in \lambda(M)$ then $q^{\langle \sigma, a \rangle + \deg D} \geq 1$. It follows that for any function $p : \lambda(M) \to \bbR_{\geq 0}$, (\ref{eqn_borel_boundary_term}) is bounded above by a constant multiple of
\[ \sum_{\substack{\sigma \in \Lambda_B^\text{pos} \\ \forall a \in M, \langle \sigma, a \rangle + \deg D < 0 \\
\forall a \in \Phi_V - M, \langle \sigma, a \rangle + \deg D \geq 0}} q^{\deg D (\sum_{a \in \lambda(M)} p(a) - |M| ) + \langle \sigma,  \sum_{a \in \lambda(M)} p(a) a - \sum_{a \in M} a  - \delta_B\rangle}\]
\[ \leq q^{\deg D (\sum_{a \in \lambda(M)} p(a) -|M |)} \sum_{\sigma \in \Lambda_B^\text{pos} }  q^{\langle \sigma, \sum_{a \in \lambda(M)} p(a) a - \sum_{a \in M} a - \delta_B \rangle}. \]
This last expression tends to 0 as $\deg D$ tends to infinity provided the function $p$ is chosen so that the following conditions are satisfied:
\begin{itemize}
\item $|M| > \sum_{a \in \lambda(M)} p(a)$.
\item Define $w(M) =  - \sum_{a \in M} a - \delta_B$ and $w(M, p) =  \sum_{a \in \lambda(M)} p(a) a - \sum_{a \in M} a - \delta_B \in X^\ast(T)_\bbR$. Then $n_i(w(M, p)) > 0$ for each $i = 1, \dots, 4$.
\end{itemize}
We show that we can find such a function $p$ simply by exhibiting one for each possible choice of $M \in \cC_0$ in the following table (the weights being labelled as in \S \ref{sec_preliminaries}):
\begin{center}
\begin{tabular}{|c|c|c|cccc|c|cccc|}
\hline
$M$ & $\lambda(M)$ & $|M|$ & \multicolumn{4}{|c|}{$2w(M)$}&  $p$ & \multicolumn{4}{|c|}{$2w(M, p)$} \\
\hline
  1  &  2, 3, 4, 5  & 1  &  1  &  1  &  1 & 1  & $(0,0,0,0)$ & 1 &1&1&1\\
  1,2  &  3,4,5  & 2   &  2  &  0  & 0  &  0& $(0.5,0.5,0.5)$ & 3.5 &0.5&0.5&0.5 \\
1,3&2,4,5&2&0&2&0&0& $(0.5,0.5,0.5)$ & 0.5 &3.5&0.5&0.5\\
1,4&2,3,5&2 &0&0&2&0&$(0.5,0.5,0.5)$&0.5 &0.5&3.5&0.5\\
1,5&2,3,4&2&0&0&0&2&$(0.5,0.5,0.5)$& 0.5 &0.5&0.5&3.5\\
1,2,3&4,5,6&3&1&1&-1&-1&$(0.5,0.5,1.5)$&0.5 &0.5&0.5&0.5\\
1,2,4&3,5,7&3&1&-1&1&-1&$(0.5,0.5,1.5)$&0.5 &0.5&0.5&0.5\\
1,2,5&3,4,8&3&1&-1&-1&1&$(0.5,0.5,1.5)$&0.5 &0.5&0.5&0.5\\
1,3,4&2,5,9&3&-1&1&1&-1&$(0.5,0.5,1.5)$&0.5 &0.5&0.5&0.5\\
1,3,5&2,4,10&3&-1&1&-1&1&$(0.5,0.5,1.5)$&0.5 &0.5&0.5&0.5\\
1,4,5&2,3,11&3&-1&-1&1&1 &$(0.5,0.5,1.5)$&0.5 &0.5&0.5&0.5\\
\hline
  \end{tabular}
\end{center}
  This shows that the equality (\ref{eqn_limit_of_boundary_term_equals_zero}  holds in case $P = B$. We now treat the four remaining cases. By symmetry, we can assume that $P$ is the standard parabolic subgroup of $G$ generated by $B$ and the root subgroup corresponding to the root $a_1$. Then the Levi quotient $L_P$ of $P$ is isogenous to $\SL_2$, and the same argument as above shows that we need to show that 
\begin{equation}\label{eqn_vanishing_boundary_term_for_parabolic} \lim_{\deg D \to \infty} \int_{ \cY_{G, P}^{< 0, M}}  \frac{|H^0(X, \cV_g(D))^\text{sf, nt}|}{|H^0(X, B_D)|} \, d \nu_G = 0 
\end{equation}
  for each $M \in \cC_0$. We observe that $\cY_{G, P}^{< 0, M}$ is non-empty only when $M$ satisfies the condition $a \in M \Rightarrow a' \in M$, where $a' \in \Phi_V$ is defined by $n_1(a') = - n_1(a)$, $n_i(a') = n_i(a)$ for $i = 2, 3, 4$. The only set $M \in \cC_0$ which satisfies this condition is $M = \{ 1, 2 \}$, so we are reduced finally to showing that the equality (\ref{eqn_vanishing_boundary_term_for_parabolic}) holds in the single case $M = \{ 1, 2 \}$. This can be proved using exactly the same trick as before.

Putting everything back together and applying Proposition \ref{prop_existence_and_equality_of_local_squarefree_densities}, we find
\[ \lim_{\deg D \to \infty} \sum_{(E, P, Q) \in \cX_D} \frac{|\Sel_2(E)^{\text{nt}}| \cdot | \Aut(E, P, Q) |^{-1} \cdot | E(K)[2] |^{-1}} { |\cX_D |} = \delta_B^{-1}\times \sum_P q^{12(1-g_X)} \delta_V \int_{\cY_{G, P}} \, d \nu_G \]
\[ = \int_{ G(\widehat{\cO}_K)} \, d \tau_G \int_{ G(K) \backslash G(\bbA_K)} \, d \mu_G = \int_{ G(K) \backslash G(\bbA_K)} \, d \tau_G = \tau(G), \]
the Tamagawa number of $G$. Since the fundamental group of $G$ is isomorphic to $\mu_2^3$, we have $\tau(G) = 8$. This completes the proof.
\end{proof}
\begin{corollary}\label{cor_average_2_selmer}
The limit \[ \lim_{\deg D \to \infty} \sum_{(E, P, Q) \in \cX_D}  \frac{|\Sel_2(E) | \cdot | \Aut(E, P, Q) |^{-1}  \cdot | E(K)[2] |^{-1}} { |\cX_D| } \]
exists and equals 12.
\end{corollary}
\begin{proof}
In view of Theorem \ref{thm_average_of_non_trivial_selmer}, we just need to show that
\[ \lim_{\deg D \to \infty} \sum_{(E, P, Q) \in \cX_D} \frac{| A_{E, P, Q} | \cdot | \Aut(E, P, Q) |^{-1}  \cdot | E(K)[2] |^{-1} } {| \cX_D |} \]
exists and equals 4. Following the discussion after Theorem \ref{thm_rational_orbits_and_rational_points}, we see that it is enough to show that the limit
\[ \lim_{\deg D \to \infty} \frac{ | \{ b \in H^0(X, B_D) \cap B^\text{rs}(K) \mid \im( \Gamma_K \to W(G, \frz_\frh(\kappa_b)) ) \neq W(G, \frz_\frh(\kappa_b))  \} | }{ | H^0(X, B_D) | } \]
exists and equals 0. This is a consequence of the Hilbert irreducibility theorem. 
\end{proof}
Finally, we prove the promised generalization of Theorem \ref{thm_average_of_non_trivial_selmer}.
\begin{theorem}\label{thm_equidistribution_of_non_trivial_selmer}
Let $f : \cY_G \to \bbR$ be a bounded function. Then we have
\[ \lim_{\deg D \to \infty} \sum_{(E, P, Q) \in \cX_D} \frac{ | \Aut(E, P, Q) |^{-1} \cdot |E(K)[2]|^{-1} \sum_{x \in \Sel_2( E)^\text{nt} } f(\inv x) } { |\cX_D| } = \int_{F \in \cY_G} f(F) \, d\tau_G. \]
\end{theorem}
\begin{proof}
Arguing as in the proof of Theorem \ref{thm_average_of_non_trivial_selmer}, we get
\[ (q-1) \sum_{(E, P, Q) \in \cX_D} \frac{ | \Aut(E, P, Q) |^{-1} \cdot |E(K)[2]|^{-1} \sum_{x \in \Sel_2( E)^\text{nt} } f(\inv x) } { | H^0(X, B_D)^\text{sf} | } = \int_{ g \in \cY_G} \frac{ | H^0(X, \cV_g(D))^\text{sf} | }{ | H^0(X, B_D)^\text{sf} |} f(F_g) \, d \nu_G \]
\[ = \sum_P \left[ \int_{\cY_{G, P}(D)^{> 2g_X - 2}}  f(F_g) \frac{|H^0(X, \cV_g(D))^\text{sf}|}{|H^0(X, B_D)|} \, d \nu_G \right. \] \[ \left. + \int_{\cY_{G, P}^\text{sp}}  f(F_g) \frac{|H^0(X, \cV_g(D))^\text{sf}|}{|H^0(X, B_D)|} \, d \nu_G + \int_{\cY_{G, P}^{< 0}} f(F_g) \frac{|H^0(X, \cV_g(D))^\text{sf}|}{|H^0(X, B_D)|} \, d \nu_G \right]. \]
Since $f$ is bounded, the same arguments as before show that the boundary terms vanish in the limit. On the other hand, the boundedness of $f$ means we can again apply the dominated convergence theorem to deduce that
\[ \lim_{\deg D \to \infty}\int_{\cY_P(D)^{> 2g_X - 2}} f(F_g)  \frac{|H^0(X, \cV_g(D))^\text{sf}|}{|H^0(\cX, B_D)|}\, d \nu_G = q^{12(1-g_X)} \delta_V \int_{\cY_{G, P}} f(F_g) \,  d \nu_G , \]
and these terms can then be regrouped to obtain the statement of the theorem.
\end{proof}

\bibliographystyle{alpha}
\bibliography{orbits}
\end{document}